\documentclass[a4paper, 11pt]{article}

\usepackage{packageList}
\usepackage{macros}
\usepackage{authblk}

\title{Convergence rate of the data-independent $P$-greedy algorithm in kernel-based approximation}

\author[1]{G. Santin \thanks{santinge@mathematik.uni-stuttgart.de, 
orcid.org/0000-0001-6959-1070}}
\author[1]{B. Haasdonk \thanks{haasdonk@mathematik.uni-stuttgart.de}}

\affil[1]{Institute for Applied Analysis and Numerical Simulation, University of Stuttgart, Germany}

\begin{document}
\maketitle
\begin{abstract}
Kernel-based methods provide flexible and accurate algorithms for the reconstruction of functions from meshless samples. A major question in the use of such 
methods is the influence of the samples locations on the behavior of the approximation, and feasible optimal strategies are not known for general 
problems. 

Nevertheless, efficient and greedy point-selection strategies are known. This paper gives a proof of the convergence rate of the data-independent 
\textit{$P$-greedy} algorithm, based 
on the application of the convergence theory for greedy algorithms in reduced basis methods. The resulting rate of convergence is shown to be near-optimal in 
the case of kernels generating Sobolev spaces.

As a consequence, this convergence rate proves that, for kernels of Sobolev spaces, the points selected by the algorithm are asymptotically uniformly 
distributed, as conjectured in the paper where the algorithm has been introduced. 
\end{abstract}

\section{Introduction}\label{sec:introduction}
We start by recalling some basic facts of kernel based approximation. Further details and a thorough treatment of the topic can be found e.g. in the 
monographs \cite{Wendland2005, Fasshauer2007, Buhmann2003, Fasshauer2015}.

On a compact set $\Omega\subset\R^d$ we consider a continuous, symmetric and strictly positive definite kernel $K:\Omega\times\Omega\to\R$. Positive 
definiteness is understood in terms of the associated \textit{kernel matrix}, i.e., for all $n\in\N$ and $\{x_1, \dots, x_n\}\subset\Omega$ pairwise distinct 
the kernel matrix $A\in \R^{n\times n}$, $A_{ij} := K(x_i, x_j)$, is positive definite.

Associated with the kernel there is a uniquely defined \textit{native space} $\ns$, that is, the 
unique Hilbert 
space of functions from $\Omega$ to $\R$ in which $K$ is the \textit{reproducing kernel}, i.e., 
\begin{enumerate}[(a)]
 \item $K(\cdot, x) \in\ns$ for all $x\in\Omega$,
 \item $(f, K(\cdot, x)) = f(x)$ for all $f\in\ns$, $x\in\Omega$.
\end{enumerate}
We used here and we will use in the following the notation $(\cdot, \cdot)$ and $\|\cdot\|$, 
without subscripts, for the inner product and norm of $\ns$.

For any given finite set $X_n := \{x_1,\dots,x_n\}\subset \Omega$ of $n$ pairwise distinct points, the interpolation of a function $f\in\ns$ on $X_n$ is well 
defined being the kernel strictly positive definite, and it coincides with the orthogonal projection $\Pi_{V(X_n)}(f)$ of $f$ into $V(X_n)$, where $V(X_n) := 
\Sp{K(\cdot, x_k), 1\leq k \leq n}$ is the $n$-dimensional subspace of $\ns$ generated by the 
kernel translates on $X_n$. We will denote by $|\cdot|$ the number of pairwise distinct elements of 
a finite set, 
i.e., $|X_n| := n$.

% To solve the interpolation problem, one needs to solve a liner system defined by the $n\times n$ \textit{kernel matrix} $A:=[K(x_i, x_j)]_{i,j=1}^n$.
Since $\Pi_{V(X_n)}(f)\in V(X_n)$, the interpolant is of the form 
$$
\Pi_{V(X_n)}(f) : = \sum_{k=1}^n \alpha_k K(\cdot, x_k),
$$ 
for some coefficients 
$\{\alpha_k\}_{k=1}^n$. To actually compute these, one imposes the interpolation conditions 
$\Pi_{V(X_N)}(f)(x_i) = f(x_i)$, $1\leq i\leq n$,  which result in the linear system 
\begin{equation}\label{eq:kernel_matrix}
A \alpha = b, 
\end{equation}
which has in fact a unique solution for all $b\in\R^n$, $b_i:=f(x_i)$, $A$ being positive definite.

The standard way to measure the interpolation error is by means of the \textit{Power Function} $P_{V(X_n)}$, which is defined in a point $x\in\Omega$ as the 
norm of the pointwise interpolation error at $x$, i.e., 
\begin{equation}\label{eq:power_definition}
P_{V(X_n)}(x) := \sup_{f\in \ns, f\neq 0} \frac{|f(x) - \Pi_{V(X_n)}(f)(x)|}{\|f\|},
\end{equation}
and it is a continuous function on $\Omega$, vanishing only on $X_n$. Among other equivalent definitions of the Power Function (e.g., by considering a 
\textit{cardinal basis} $\{\ell_k\}_{k=1}^n$ of $V(X_n)$, i.e., $\ell_k(x_i) = 
\delta_{ki}$), the present one is 
easier to generalize to the setting considered in Section \ref{sec:power_and_width}. 
From the definition, it is immediate to see that bounds on the 
maximal value of the Power Function in $\Omega$ provide uniform bounds on the interpolation error as
\begin{equation}\label{pf_error}
\left\|f - \Pi_{V(X_n)}(f)\right\|_{L_{\infty}(\Omega)}\leq  \left\|P_{V(X_n)}\right\|_{L_{\infty}(\Omega)} \|f\|,\; f\in\ns.
\end{equation}

It is thus of interest to find and characterize point sets $X_n$ which guarantee a small value of $\left\|P_{V(X_n)}\right\|_{L_{\infty}(\Omega)}$, and the 
reason is twofold. If one is free to consider any point in $\Omega$, selecting good points means to 
construct an optimal or suboptimal discretization of the 
set with respect to kernel approximation. On the other hand, if a set of data points $X_N\subset 
\Omega$ 
is provided (e.g., the location of the measurements 
coming from an 
application), it is often desirable to be able to select a subset $X_n\subset X_N$, $n\ll N$, of the full data to reconstruct a sparse approximation of the 
unknown function, where sparsity is understood both in terms of the underlining linear 
system and in a functional sense. Indeed, selecting $X_n\subset X_N$ means to solve the system 
\eqref{eq:kernel_matrix} with respect to the submatrix defined by the small point set, which can be used to define a  
sparse approximation of the full kernel matrix. On the other hand, the resulting interpolant (or 
model of the data) is given by an expansion of only $n$ out of $N$ kernel translates, and this 
means that its evaluation is cheaper and more suitable to be used as a surrogate model of the data.

Although feasible selection criteria to construct an optimal set $X_n$ are generally not known, 
different greedy techniques have been presented to construct 
\textit{near optimal} points (see \cite{DeMarchi2005, Pazouki2011, Wirtz2013}). They are based on the idea 
that it is possible to construct good sequences of nested sets of points starting from the empty set 
$X_0:=\emptyset$, and iteratively increasing the set as $X_n:=X_{n-1}\cup\{x_n\}$ by adding a new 
point chosen to maximize a certain indicator. The resulting algorithms all share the same 
structure, while the choice of the point selection criteria is different. 

Among various 
methods, we will concentrate here on the so-called 
\textit{$P$-greedy} algorithm which has been introduced in \cite{DeMarchi2005}. It is a 
\textit{data independent} algorithm, meaning that the selection of the points is made by only 
looking at $K$ and $\Omega$ (and possibly $X_N$), but not at the samples of a particular 
function $f\in\ns$, and it thus produces point sets which provide uniform approximation errors for 
\textit{any} function $f\in\ns$. To be more precise, the selection criterion picks at every 
iteration the point in $\Omega\setminus X_{n-1}$ which maximizes the Power Function 
$P_{V(X_{n-1})}$. By adding this point to the set $X_{n-1}$, the new Power Function $P_{V(X_n)}$ 
vanishes at $x_n$, and indeed, as we will explain later, $\|P_{V(X_n)}\|_{L_{\infty}(\Omega)}\leq 
\|P_{V(X_{n-1})}\|_{L_{\infty}(\Omega)}$.
%, and this reduction in the Power Function is locally (i.e., at step $n$) optimal. %This in turn guarantees a locally 
% optimal improvement of the error bound 
% \eqref{pf_error} at every step, for any $f\in\ns$.

The goal of this paper is to prove that the points produced by this algorithm are indeed 
near-optimal, meaning that they have the same asymptotic decay of the best known, non greedy point 
distributions. In particular, in the paper \cite{DeMarchi2005}, the authors considered the case of 
translational invariant  and Fourier transformable kernels on domains satisfying an interior cone 
condition, for which the asymptotic decay of the Power Function is well understood for certain 
point distributions. We remark that Radial Basis Functions are instances of such kernels. In this 
setting, in the paper \cite{DeMarchi2005} the following decay rate for the $P$-greedy algorithm has been shown, 
which is, up to our knowledge, the currently sharpest known convergence statement. 
\begin{theorem}\label{th:pgreedy_demscwe}
If $\Omega$ is compact in $\R^d$ and satisfies an interior cone condition, and $K\in\mathcal C^2(\Omega_1\times\Omega_1)$, with $\Omega\subset\Omega_1$, 
$\Omega_1$ compact and convex, then the point sets $\{X_n\}_n$ selected by the $P$-greedy algorithm have Power Functions such that, for any 
$n\in\N$, 
$$
\|P_{V(X_n)}\|_{L_{\infty}(\Omega)} \leq c n ^{-\frac1d},
$$
for a constant $c$ not depending on $n$.
\end{theorem}
The proof of this theorem requires that $K\in\mathcal C^2$ on a suitable set, and our bound indeed is similar to the present 
one under the same 
assumptions, while it will improve it when the additional smoothness of the kernel is taken into account. This 
refined error bound allows also to prove that the selected points, for certain kernels, are 
asymptotically uniformly distributed.

The paper is organized as follows. In Section \ref{sec:kernel-basics} we review the known estimates 
on the decay of the Power Function and give further details on the $P$-greedy 
algorithm. Section \ref{sec:power_and_width} is devoted to provide a connection between Kolmogorov widths and maximization of the Power Function. This 
connection allows to employ the theory of \cite{Binev2011, DeVore2013} in Section \ref{sec:convergence} to prove the main results of this paper. Finally, in 
Section \ref{sec:numerics} we present some numerical experiments which verify the expected rates of convergence. 

\begin{rem}
We remark that, although our analysis is presented for the reconstruction of scalar-valued 
functions, it applies also to the vector-valued case when using product spaces: namely, as pointed out in \cite{Wirtz2013}, for $q\geq 1$ it is possible to 
use kernel methods to reconstruct functions $f:\Omega\to 
\R^q$  simply by  considering $q$ copies of $\ns$, i.e., the product space
$$
\mathcal H_K(\Omega)^q:=\{f:\Omega\to \R^q, f_j\in\ns\}
$$
equipped with the inner product 
$$
(f, g)_q := \sum_{j=1}^q (f_j, g_j),
$$
where, to avoid having $q$ different expansions, one for each component, one can do the further 
assumption that a unique subspace $V(X_n)$ is used for every 
component. In this context, the present discussion on the $P$-greedy algorithm is directly 
applicable without modifications.
\end{rem}

\section{Power Function and the P-greedy algorithm}\label{sec:kernel-basics}
To assess the convergence rate of the $P$-greedy algorithm, we compare it with the known estimates on the decay of the Power Function. The following 
bounds apply to the notable case of translational invariant kernels, for which the behavior of the Power Function is well understood.

To be more precise, we assume from now on that there exists a function $\Phi:\R^d\to \R$ 
such that $K(x, y) := \Phi(x-y)$, and that $\Phi$ has a continuous Fourier transform $\hat\Phi$ on $\R^d$. We further assume that 
$\Omega$ satisfies an interior cone condition. 
Under these assumptions, the decay of $\|P_{V(X_n)}\|_{L_{\infty}(\Omega)}$ can be related to the smoothness of $\Phi$ (hence of $K$) and to the \textit{fill 
distance}
$$
h_{X_n, \Omega}:= \sup_{x\in\Omega}\min_{x_j \in X_n} \|x-x_j\|_2,
$$
where $\|\cdot\|_2$ is the Euclidean norm on $\R^d$. The next theorem summarizes such estimates (see \cite{Schaback1995}). We remark that the two cases (a) and 
(b) are substantially different. The first one regards kernels for which there exist $c_{\Phi}, C_{\Phi}>0$ and $\beta\in\N$, $\beta>d/2$, such that
$$
c_{\Phi}\left(1 + \|\omega\|_2^2\right)^{-\beta} \leq \hat\Phi(\omega) \leq C_{\Phi}\left(1 + \|\omega\|_2^2\right)^{-\beta},
$$ 
shortly $\hat \Phi(\omega) \sim (1 + \|\omega\|_2^2)^{-\beta}$, in which case $K\in\mathcal C^{\beta}$ and $\mathcal H_K(\R^d)$ is norm equivalent to the 
Sobolev space 
$W_2^{\beta}(\R^d)$. The second one applies to kernels of infinite smoothness, such as the Gaussian kernel. We will use the notion of kernels of \textit{finite} 
or \textit{infinite} smoothness to indicate precisely these two cases.

\begin{theorem}\label{th:power_fill_distance}
Under the assumptions on $K$ and $\Omega$ as above, we have the following cases, for suitable constants $\hat c_1, \hat c_2, \hat c_3$ not depending on $X_n$.
\begin{enumerate}[(a)]
 \item If $K$ has finite smoothness $\beta\in\N$,  
 $$
 \|P_{V(X_n)}\|_{L_{\infty}(\Omega)} \leq \hat c_1 h_{X_n, \Omega}^{\beta - d/2}.
 $$
 \item If $K$ is infinitely smooth,  
 $$
 \|P_{V(X_n)}\|_{L_{\infty}(\Omega)} \leq \hat c_2 \exp(-\hat c_3 / h_{X_n, \Omega}).
 $$
\end{enumerate}
\end{theorem}
In particular, one can look at \textit{asymptotically uniformly distributed points} in $\Omega$, i.e., sequences $\{X_n\}_n$  of points such that $h_{X_n, 
\Omega} \leq 
c n ^{-1/d}$, for a 
constant $c\in \R$ not depending on $n$. The above estimates can then be written only in terms of $n$.
\begin{cor}\label{lemma:pf_bounds}
In the same setting as in Theorem \ref{th:power_fill_distance}, there exists sequences $\{X_n\}_n$ of points in $\Omega$ and constants $c_1, c_2, c_3$, whose 
Power Function behaves as follows for $n\in \N$.
\begin{enumerate}[(a)]
 \item If $K$ has finite smoothness $\beta\in\N$,
 $$
 \|P_{V(X_n)}\|_{L_{\infty}(\Omega)} \leq c_1 n^{-\frac{\beta}{d} + \frac12}.
 $$
 \item If $K$ is infinitely smooth,  
 $$
 \|P_{V(X_n)}\|_{L_{\infty}(\Omega)} \leq c_2 \exp(-c_3 n ^{1/d}).
 $$
\end{enumerate}
\end{cor}
To refer to a convergence of the Power Function as $n$ increases, and in particular to one of the above rates, we will write 
$$
 \|P_{V(X_n)}\|_{L_{\infty}(\Omega)} \leq \gamma_n\;\; \mbox{ with } \lim_{n\to\infty}\gamma_n = 0.
$$

\subsection{The $P$-greedy algorithm}
We describe here in some more detail the structure of the algorithm, and provide some details on 
its implementation.
The algorithm starts with an empty set $X_0 := \emptyset$ and with the zero subspace $V(X_0) 
:=\{0\}$, and it constructs a sequence of nested point sets 
$$
X_0\subset X_1\subset \dots\subset X_n\subset\dots\subset \Omega,
$$
by sequentially adding a new point, i.e., $X_n := X_{n-1}\cup \{x_n\}$. A sequence of nested linear 
subspaces 
$$
V(X_0)\subset V(X_1)\subset \dots\subset V(X_n)\subset\dots\subset \ns,
$$
is associated to the point sets, and for each of them a Power Function $P_{V(X_n)}$ can be defined. For $n = 0$, definition \eqref{pf_error}
gives 
$P_{V(X_0)}:=\sqrt{K(x,x)}$, since
$$
|f(x) - \Pi_{V(X_0)}(f)(x)| = |f(x)| = |(f, K(\cdot, x))| \leq \|K(\cdot,x)\| \|f\| = \sqrt{K(x,x)} \|f\|,
$$ 
and equality is obtained for $f:=K(\cdot, x)$.

The points are chosen by picking the current maximum on $\Omega\setminus X_n$ of the the $n$-th Power Function, i.e., 
\begin{align*}
x_1 &:= \argmax_{x\in\Omega} P_{V(X_0)}(x) = \sqrt{K(x,x)},\\
x_n &:= \argmax_{x\in\Omega\setminus X_{n-1}} P_{V(X_{n-1})}(x).
\end{align*}
In particular, the choice of the first point is arbitrary for a translational invariant kernel, and in 
general all the points are not uniquely defined, being the 
maxima of the Power Function not necessarily unique. 

This $P$-greedy algorithm has an efficient implementation in terms of the \textit{Newton basis} (see
\cite{Muller2009}), which allows to easily deal with nested subspaces and the corresponding 
orthogonal projections. Namely, assuming to have a sequence $\{X_n\}_n$ of nested point sets, the construction of the Newton basis is a Gram-Schmidt 
procedure over the set of the kernel translates at these points, and the resulting set of functions $\{v_k\}_{k=1}^n$ is indeed an orthonormal basis of 
$V(X_n)$, with the further property that $\Sp{v_k, 1\leq k\leq n} = V(X_n)$. In particular, the basis does not need to be recomputed when a new point is added.
We remark that this construction can be efficiently implemented by a matrix-free (i.e., only one column at a time is 
computed) partial LU-decomposition of the kernel matrix, with the pivoting rule given by the present selection criteria (see \cite{Pazouki2011}).

As mentioned in Section \ref{sec:introduction}, the $P$-greedy selection strategy guarantees that the Power Function decreases. To 
prove this fact, we first recall the following characterization of the Power Function, which we prove for completeness.
\begin{lemma}\label{lemma:pf_and_riesz}
For any subspace $V(X_n)\subset \ns$ and $x\in\Omega$, the Power Function has the representation
\begin{equation}
P_{V(X_n)}(x) = \|K(\cdot, x) - \Pi_{V(X_n)}(K(\cdot, x))\|.
\end{equation}
\end{lemma}
\begin{proof}
Let $f\in\ns$, $\|f\|\leq 1$ and consider an orthonormal basis $\{v_k\}_k$ of $V(X_n)$. We define here $v_x:=K(\cdot,x)$ for simplicity of notation. The 
interpolation error for $f$, measured at $x\in\Omega$, is 
\begin{align*}
&f(x) - \Pi_{V(X_n)}(f) (x) = (v_{x}, f) - \left(v_{x}, \sum_{k=1}^n (f, v_k) v_k \right) \\
&= (v_{x}, f) 
- \sum_{k=1}^n (f, v_k) 
\left(v_{x}, v_k \right) 
 = (v_{x}, f) - \left(\sum_{k=1}^n (v_{x}, v_k)  v_k, f \right) \\
&\leq \left\|v_{x} - \sum_{k=1}^n 
(v_{x}, v_k)  v_k \right\| 
\|f\|
 = \|v_{x} - \Pi_{V(X_n)}(v_{x})\| \|f\|,
\end{align*}
thus $P_{V(X_n)}(x) \leq \|v_{x} - \Pi_{V(X_n)}(v_{x})\|$, and the equality is actually reached by taking 
$$
f_x:= \frac{v_{x} - \Pi_{V(X_n)}(v_{x})}{\|v_{x} - \Pi_{V(X_n)}(v_x)\|}.
$$
\end{proof}

It is then clear that, for any orthonormal basis $\{v_k\}_{k=1}^n$ of $V(X_n)$, we have 
\begin{equation}\label{eq:power_computation}
P_{V(X_n)}(x)^2 = \|K(\cdot, x) - \Pi_{V(X_n)}(K(\cdot, x))\|^2 = K(x, x) - \sum_{k=1}^n v_k(x) ^2,
\end{equation}
and in particular, by using the Newton basis as an orthonormal basis,
\begin{equation}\label{eq:power_update}
P_{V(X_n)}(x)^2 = P_{V(X_{n-1})}(x)^2 - v_n(x) ^2.
\end{equation}

This means that the Power Function is decreasing whatever the choice of $x_n\in\Omega\setminus X_n$ is, i.e., we have 
$$
\|P_{V(X_{n})}\|_{L_{\infty}(\Omega)}\leq \|P_{V(X_{n-1})}\|_{L_{\infty}(\Omega)}.
$$
% Moreover, by selecting as $x_n$ the point where the current Power Function attains its maximum, the algorithm maximizes at each step this reduction 
% effect. 

% The first basis element is defined as
% $$
% v_1 := \frac{K(\cdot, x_1) }{\|K(\cdot, x_1)\|} = \frac{K(\cdot, x_1) }{\sqrt{K(x_1, x_1)}}.
% $$
% At the step $n$ the kernel translate $K(\cdot, x_n)$ is projected on the first $n-1$ basis elements 
% to define the function
% $$
% \tilde v_n := K(\cdot, x_n) - \sum_{k = 1}^{n-1} (v_k, K(\cdot, x_n)) v_k = K(\cdot, x_n) - \sum_{k 
% = 1}^{n-1} v_k(x_n) v_k
% $$
% having squared norm 
% $$
% \|\tilde v_n\|^2 := K(x_n, x_n) - \sum_{k = 1}^{n-1} v_k(x_n) ^2 = P_{n-1}^2(x_n).
% $$
% The $n$-th Newton basis is then defined as 
% $$
% v_n := \frac{\tilde v_n}{\|\tilde v_n\|} = \frac{K(\cdot, x_n) - \sum_{k = 1}^{n-1} v_k(x_n) 
% v_k}{P_{n-1}(x_n)},
% $$
% and we have indeed $(v_k, v_h) = \delta_{hk}$ for $1\leq h,k\leq n$, so the Newton basis is in fact 
% an orthonormal basis of $V(X_n)$. Notice that $v_n (x_n) = 
% P_{n-1}(x_n)$.

\section{Power Function and Kolmogorov width}\label{sec:power_and_width}
We can now provide a connection between the Power Function and the \textit{Kolmogorov $n$-width} 
of a particular compact subset of $\ns$. Recall 
that for a subset $\mathcal V\subset \ns$ the Kolmogorov $n$-width of $\mathcal V$ in the Hilbert 
space $\ns$ is defined as (see e.g. \cite{Pinkus1985})
\begin{equation*}
d_n(\mathcal V, \ns) : = \inf_{\stackrel{V_n\subset\ns}{\dim(V_n) = n}} \sup_{f\in \mathcal V} \|f - 
\Pi_{V_n}(f)\| = 
\inf_{\stackrel{V_n\subset\ns}{\dim(V_n) = n}} 
E(\mathcal V, V_n),
\end{equation*}
where the term $ E(\mathcal V, V_n)$ represents the worst-case error in approximating elements of 
$\mathcal V$ by means of elements of the liner subspace $V_n$. One has 
$d_n(\mathcal V, \ns)\leq \sup_{f\in \mathcal V} \|f\|$, and in particular we allow $d_n(\mathcal V, 
\ns) = +\infty$ for unbounded sets.

In order to analyze the connection between $P_{V_n}$ and $d_n$, we recall that a generalized interpolation operator can be 
defined for any 
$n$-dimensional linear subspace $V_n$ of $\ns$, not necessarily in the form 
$V(X_n)$, simply by considering the orthogonal projection operator $\Pi_{V_n}:\Omega \to V_n$ as a 
generalized interpolation operator. 
A generalized Power Function can be defined also in this case by directly 
using the definition \eqref{eq:power_definition} (see \cite{Santin2016a}), and Lemma \ref{lemma:pf_and_riesz} still holds with the same proof, i.e., 
\begin{equation}
P_{V_n}(x) = \|K(\cdot, x) - \Pi_{V_n}(K(\cdot, x))\|\;\fa x\in\Omega.
\end{equation}

With this characterization at hand, it comes easy to provide a connection with Kolmogorov widths. 
Namely, for any subset $\tilde\Omega\subseteq \Omega$, we 
can define the subset $\mathcal V(\tilde\Omega) : = \{K(\cdot, x), x\in \tilde\Omega\}\subset \ns$. 
Thanks to \eqref{eq:power_computation} and Lemma \ref{lemma:pf_and_riesz}, it is clear that
\begin{equation}\label{eq:error_and_power}
E(\mathcal V(\tilde\Omega), V_n) = \sup_{f\in \mathcal V(\tilde\Omega)} \|f - \Pi_{V_n}(f)\| = 
\sup_{x\in \tilde\Omega} P_{V_n}(x) = 
\|P_{V_n}\|_{L_{\infty}(\tilde\Omega)}.
\end{equation}

We have then the following. 
\begin{lemma}\label{lemma:p_and_k}
Let $\tilde\Omega \subseteq \Omega$.
If there exist point sets $\{X_n\}_n\subset \Omega$ , each of $n$ pairwise distinct points, and a 
sequence $\{\gamma_n\}_n\subset \R$ such that
$ \|P_{V(X_n)}\|_{L_{\infty}(\tilde\Omega)} \leq \gamma_n$, then 
\begin{equation}\label{eq:kol_decay}
d_n(\mathcal V(\tilde\Omega), \ns) \leq \gamma_n,
\end{equation}
and $\mathcal V(\tilde\Omega)$ is compact in $\ns$ if $\lim_{n\to\infty}\gamma_n = 0$. In 
particular, in the setting of Corollary \ref{lemma:pf_bounds},
\begin{enumerate}[(a)]
 \item\label{item:p_and_k_a} if $K$ has finite smoothness $\beta\in\N$,
 $$
d_n(\mathcal V(\tilde\Omega), \ns) \leq c_1 n ^{-\frac{\beta}{d} + \frac12}, 
$$
\item\label{item:p_and_k_b} if $K$ is infinitely smooth,
$$
d_n(\mathcal V(\tilde\Omega), \ns) \leq c_2 \exp(-c_3\ n ^{1/d}), 
$$
\end{enumerate}
and in both cases $\mathcal V(\tilde\Omega)$ is compact in $\ns$. 
\end{lemma}
\begin{proof}
From \eqref{eq:error_and_power}, and from the definition of the Kolmogorov width, one has 
\begin{align*}
d_n(\mathcal V(\tilde\Omega), \ns) =& \inf_{\stackrel{V_n\subset\ns}{\dim(V_n) = n}} 
\|P_{V_n}\|_{L_{\infty}(\tilde\Omega)} \leq 
\inf_{\stackrel{X_n\subset\Omega}{|X_n| = n}} \|P_{V(X_n)}\|_{L_{\infty}(\tilde\Omega)}\\
\leq&\inf_{\stackrel{X_n\subset\Omega}{|X_n| = n}} \|P_{V(X_n)}\|_{L_{\infty}(\Omega)},
\end{align*}
where the first inequality follows from restricting the set over which the infimum is computed, and 
the second one by considering the $L_{\infty}$-norm over the larger set $\Omega\supseteq 
\tilde\Omega$. 

Now, for any  $X_n\subset\Omega$ with $|X_n| = n$, $\|P_{V(X_n)}\|_{L_{\infty}(\Omega)}$ is an upper 
bound on the last term of the above inequalities, being 
$X_n$
non necessarily optimal. In particular this holds for the sequence of points $\{X_n\}_n$ of the 
statement, with Power Functions bounded by a sequence 
$\{\gamma_n\}_n$, 
which proves \eqref{eq:kol_decay}. Moreover, according to \cite[Prop. 1.2]{Pinkus1985}, a set 
$\mathcal V\subset\ns$ is compact if and only if it is bounded 
and 
$d_n(\mathcal V, \ns) \to 0$ as $n\to \infty$. But this is the case for $\mathcal V:= \mathcal 
V(\tilde\Omega)$ whenever $\lim_{n\to\infty}\gamma_n = 
0$, since 
$$
\sup\ \{\|f\|, f\in \mathcal V(\tilde\Omega)\} = \sup\ \{\sqrt{K(x,x)}, x\in \tilde\Omega\} = \sqrt{\Phi(0)} \in \R.
$$
In particular, by using the rates of convergence of Corollary \ref{lemma:pf_bounds} one gets the estimates \eqref{item:p_and_k_a} and \eqref{item:p_and_k_b} 
for different kernel smoothness. 
\end{proof}

\begin{rem}\label{rem:omega_tilda}
It is clear that this result holds for $\tilde\Omega = \Omega$, and this is indeed the most interesting case. Nevertheless, in actual computations one 
has generally never access to $\Omega$, but only to a subset $\tilde\Omega$, being it an arbitrary discretization required for numerically representing 
the continuous set, or a large set of data $\tilde\Omega=X_N$ coming from an application. In this case, also the optimization required by the greedy algorithm 
is performed on 
$\tilde\Omega$, and not on $\Omega$. By explicitly considering this restricted set in the above Kolmogorov width, we will be able to give exact bound on the 
convergence of the $P$-greedy algorithm when executed over $\tilde\Omega$, as will be explained in the next Section.
\end{rem}

\section{Convergence rate of the $P$-greedy algorithm}\label{sec:convergence}
The discussion of the previous Section is what we need to provide a connection to the theory of greedy algorithms developed in the 
papers \cite{Binev2011, DeVore2013}. Indeed, the $P$-greedy algorithm can be rewritten in terms of the so-called \textit{strong}
greedy algorithm of these papers as follows. 

We consider a target compact set 
$\mathcal 
V(\tilde\Omega)\subset \ns$, and, for $n\geq 1$, we select a sequence of functions $\{f_k\}_k\subset\mathcal V(\tilde\Omega)$ such 
that $\Sp{f_k, 1\leq k\leq n}$ is an approximation of $\mathcal V$. The first element 
$f_1$ is defined as
$$
f_1 := \argmax_{f\in\mathcal V(\tilde\Omega)} \|f\| = \argmax_{x\in\tilde\Omega} \sqrt{K(x,x)}.
$$
Assuming $f_1, \dots, f_{n-1}$ has been selected and $V_{n-1} := \Sp{f_1,\dots, f_{n-1}} = 
\Sp{K(\cdot, x_k), x_k\in X_{n-1}}$, the next element is
$$ 
f_n := \argmax_{f\in\mathcal V(\tilde\Omega)} E(f, V_{n-1}) = \argmax_{x\in\tilde\Omega} P_{V_{n-1}}(x) = \argmax_{x\in\tilde\Omega\setminus 
X_{n-1}} P_{V_{n-1}}(x),
$$
where we used in the last step the fact that $P_{V(X_{n-1})} =0 $ on $X_{n-1}$. 
It is clear that the present algorithm is exactly the $P$-greedy algorithm. Observe also that the orthonormal system $\{f_k^*\}_k$ obtained in the cited 
papers by Gram-Schmidt orthogonalization of $\{f_n\}_n$ is precisely the Newton basis $\{v_n\}_n$. 

In this case, thanks to the compactness of $\mathcal V(\tilde\Omega)$, we can use the estimates of \cite[Corollary 3.3]{DeVore2013}, which are in fact bounds 
on $\max_{f\in\mathcal V(\tilde\Omega)} E(f, V_{n})$, i.e., on 
$\|P_{V(X_{n-1})}\|_{L_{\infty}(\tilde\Omega)}$, in terms of $d_n(\mathcal V(\tilde\Omega), \ns)$.
In our case they read as follows.

\begin{theorem}\label{th:final_bound}
Assume $K$, $\Omega$ satisfy the hypothesis of Corollary \ref{lemma:pf_bounds}. The $P$-greedy algorithm applied to 
$\tilde\Omega \subseteq \Omega$ gives point sets $X_n\subseteq \tilde\Omega$ with the 
following decay of the Power Function.
\begin{enumerate}[(a)]
 \item If $K$ has finite smoothness $\beta\in\N$,
$$
\|P_{V(X_n)}\|_{L_{\infty}(\tilde\Omega)} \leq \hat{c_1} n ^{-\frac{\beta}{d} +\frac12}.
$$
\item If $K$ has infinitely many smooth derivatives,
$$
\|P_{V(X_n)}\|_{L_{\infty}(\tilde\Omega)}  \leq \hat{c_2}  \exp(-\hat{c_3} n ^{1/d}).
$$
\end{enumerate}
The constants $\hat{c_1}$, $\hat{c_2}$, $\hat{c_3}$ do not depend on $n$ and can be computed as
$$
\hat{c_1}:=c_1 2^{\frac{5\beta}{d} - \frac32}, \quad \hat{c_2} : = \sqrt{2 c_2}, \quad 
\hat{c_3} := 2 ^{-1 - \frac{2}{d}} c_3.  
$$

\end{theorem}

\begin{rem}\label{rem:sobolev_optimal}
In the case (a) of the above theorem, something more can be deduced on the quality of the approximation provided by the $P$-greedy algorithm. Indeed, in this 
case the native space on $\Omega = \R^d$ is norm-equivalent to the Sobolev space $W_2^{\beta}(\Omega)$ and in these spaces the behavior of the best 
approximation is well understood. Indeed, denoting as $B_1\subset\ns$ the unit ball in the native space and by $\Pi_{L_2, V_n}$ the $L_2(\Omega)$-orthogonal 
projection into a linear subspace $V_n\subset L_2(\Omega)$, we can consider the Kolmogorov width 
$$
d_n(B_1, L_2(\Omega)) : = \inf_{V_n\subset L_2(\Omega)}\sup_{f\in B_1} \|f - \Pi_{L_2, V_n}(f)\|_{L_2(\Omega)},
$$
which is known to behave (see \cite{Jerome1970}) as 
$$
c n^{-\beta/d}\leq d_n(B_1, L_2(\Omega)) \leq C n^{-\beta/d}, \;c, C >0.
$$
Moreover, it has been proven in \cite{Schaback2002a} that the same rate (in fact precisely the same value) can be obtained by considering subspaces $V_n\subset 
\ns$ and the $\ns$-orthogonal projection $\Pi_{V_n}$ (the same one we used so far in this paper), i.e., 
$$
\kappa_n(B_1, L_2(\Omega)) : = \inf_{V_n\subset \ns}\sup_{f\in B_1} \|f - \Pi_{V_n}(f)\|_{L_2(\Omega)} = d_n(B_1, L_2(\Omega)).
$$
Unfortunately, the above infimum is reached by considering a subspace generated by eigenfunctions of a particular integral operator, which are not known in 
general (see e.g. 
\cite{Santin2016a}). Nevertheless, again in the paper \cite{Schaback2002a} it has been observed that standard kernel-based approximation can reach almost the 
same asymptotic order of convergence in a bounded set $\Omega\subset\R^d$. Indeed, by considering an asymptotically uniformly distributed point 
sequence $\{X_n\}_n\subset \Omega$, Corollary \ref{lemma:pf_bounds} and the error bound \ref{pf_error} give 
\begin{align*}
& \sup_{f\in B_1} \|f - \Pi_{V(X_n)}(f)\|_{L_2(\Omega)} \leq  \sup_{f\in B_1} \|f\| \|P_{V(X_n)}\|_{L_2(\Omega)}\\
& \leq  \mbox{ meas}(\Omega)^{1/2} \|P_{V(X_n)}\|_{L_{\infty}(\Omega)} \leq \mbox{ meas}(\Omega)^{1/2} c_1 n^{-\frac{\beta}{d} + \frac12},
\end{align*}
where $\mbox{ meas}(\cdot)$ is the Lebesgue measure.

Thanks to Theorem \ref{th:final_bound}, this asymptotically near-optimal rate of convergence in Sobolev spaces can be reached also by greedy techniques. 
Moreover, we will see in Section \ref{sec:numerics} that the actual convergence of the $P$-greedy algorithm seems to be in fact  of rate $n^{-\beta/d}$, and 
not only $n^{-\beta/d + 1/2}$ as proven here.
\end{rem}

\subsection{Distribution of the selected points}
The previous result has also some consequence on the distribution of the points selected by the 
$P$-greedy algorithm. When the algorithm was introduced in \cite{DeMarchi2005}, the Authors
noticed that the point were placed in an asymptotically uniform way inside $\Omega$, and in they also 
proved the following result.
\begin{theorem}\label{th:small_error_small_h}
Assume $K$ and $\Omega$ satisfy the same assumptions as in Theorem \ref{th:power_fill_distance}, with $\Phi(\omega)\sim (1 + 
\|\omega\|_2^2)^{\beta}$, $\beta>d/2$. Then for any 
$\alpha> \beta$, there exist a constant $M_{\alpha}>0$ such that, if $\varepsilon_n >0$ and $X_n\subset\Omega$ satisfy
$$
\|f - \Pi_{V(X_n)}\|_{L_{\infty}(\Omega)} \leq \varepsilon_n \|f\| \;\fa f\in\ns,
$$
then
$$
h_{X_n,\Omega}\leq M_{\alpha} \varepsilon_n^{1/(\alpha - d/2)}.
$$
\end{theorem}
Unfortunately, the rate of convergence of Theorem \ref{th:pgreedy_demscwe} was not enough to 
conclude that the points are asymptotically uniformly distributed, which is instead possible with 
the bounds of Theorem \ref{th:final_bound}.
\begin{cor}\label{cor:fill_distance}
Under the same assumptions of the previous Theorem, there exists a constant $c > 0$ such 
that, for any $n\in\N$, the sets $\{X_n\}_n$ selected by the $P$-greedy algorithm satisfy 
\begin{equation*}
h_{{X_n}, {\Omega}} \leq c n^{-\frac1d (1 - \varepsilon)},
\end{equation*}
for any $\varepsilon \in(0, 1)$, where $c$ is independent of $n$.
\end{cor}
\begin{proof}
In the present assumptions we have from Theorem \ref{th:final_bound} $\varepsilon \leq \hat{c_1} n ^{-\frac{\beta}{d} +\frac12}$.
Theorem \ref{th:small_error_small_h} then implies that, for all $\alpha> \beta$, 
$$
h_{X_n,\Omega}\leq M_{\alpha} \left(\hat{c_1} n ^{-\frac{\beta}{d} +\frac12}\right) ^{\frac{1}{\alpha - d/2}},
$$
and for any $\varepsilon \in(0, 1)$ there is an $\alpha>\beta$ such that the exponent can be written as follows
$$
\left(-\frac{\beta}{d} +\frac12\right) \left(\frac{1}{\alpha - d/2}\right) = -\frac1d\left(\frac{\beta - d/2}{\alpha-d/2}\right) =  -\frac1d (1 - \varepsilon).
$$
\end{proof}

We remark that the above result does not apply in the case of infinitely smooth kernels. On one side, 
the proof of Theorem \ref{th:small_error_small_h} uses tools which are related to Sobolev spaces, 
hence to kernels of finite smoothness. On the other hand, one could not expect a decay of the fill 
distance with exponential speed with respect to the number of points. Nevertheless, it is plausible 
to expect that also for kernels of this kind an algebraic convergence of the fill distance is 
possible, even if it is not clear with what rate.

\section{Numerical experiments}\label{sec:numerics}
We test in this Section the theoretical rates obtained in Theorem \ref{th:final_bound} for kernels of different smoothness and in different space dimensions.

In order to ensure the validity of the hypothesis on $\Omega$, in all the following experiments we consider as a base domain the unit ball $\Omega:=\{x\in\R^d, 
\|x\|_2\leq 1\}$, for $d=1, 2, 3$. Furthermore, to implement numerical calculations $\Omega$ is represented by a discretization $\tilde\Omega\subset\Omega$, 
obtained by intersecting a uniform grid in $[-1, 1]^d$ with the unit ball. The grids have respectively $10^4$ ($d = 1$), $114^2$ ($d = 2$), $28^3$ ($d = 3$)
points, so that the resulting number of points of $\tilde\Omega$ is approximately $10^4$. 
The point selection, and both the computation of the supremum norm and of the fill distance are performed on this discretized set. We point out that this choice 
of $\tilde\Omega$ is somehow arbitrary, but it is justified in view of Remark \ref{rem:omega_tilda}.

As kernels we consider radial basis functions which satisfy the requirements of the convergence results, namely the Gaussian 
kernel $G$ defined by $\Phi(r) := \exp(-(\varepsilon r) ^2)$, as an infinitely smooth kernel, and the Wendland kernels $W_{\beta, d}$ for  $\beta = 2, 3$, 
as kernels of finite smoothness $\beta$ (see \cite{Wendland1995a}). We consider unscaled version 
of the kernels, i.e., in all the experiments the shape parameter $\varepsilon$ is fixed to the 
value $\varepsilon = 1$.

The $P$-greedy algorithm is applied via a matrix-free implementation of the Newton basis, based on \cite{Muller2009, Pazouki2011}. The code can be found 
on the website of G. Santin\footnote{\url{http://www.mathematik.uni-stuttgart.de/fak8/ians/lehrstuhl/agh/orga/people/santin/index.en.html}}. The algorithm is 
stopped by means of a tolerance of $\tau = 10 ^{-15}$ on the maximal value of the square of the Power Function on $\tilde \Omega$, 
or a maximum expansion size of $n = 1000$.  We remark that the present implementation actually computes the square of the Power Function via the formula 
\eqref{eq:power_update}, so numerical cancellation can happen when $\|P_{V(X_n)}\|_{L_{\infty}(\tilde\Omega)}^2$ is close to the machine precision. We remark 
that for some class of kernels it is possible to employ a more stable and accurate computation method for the Power Function (see \cite[Section 
14.1.1]{Fasshauer2015}), even if it is not clear if and how it applies to an iterative computation like the present one.

The numerical decay rate of the Power Function for the Gaussian kernel are presented in Figure \ref{fig:gauss_pf}, and the experiments confirm the expected 
decay rate of Theorem \ref{th:final_bound}. The coefficients $\hat{c}_2$, $\hat{c}_3$ are estimated numerically, and are reported in Table 
\ref{tab:gauss_coeffs}. 
\begin{table}[!hbt]
\centering
\begin{tabular}
{|c|c|c|c|}
\hline
 &$d = 1$&$d = 2$&$d = 3$\\
 \hline
 $\hat c_2$&$3.47  $&$5.10 $& $6.37$\\
 \hline
 $\hat c_3$&$1.22$&$ 1.80$&$2.31$\\
\hline
\end{tabular}
\caption{Estimated coefficients for the decay rate of the Power Function with the Gaussian kernel.}\label{tab:gauss_coeffs}
\end{table}

\begin{figure}[!hbt]
\centering
\begin{tabular}{ccc}
\includegraphics[width=0.3 \textwidth]{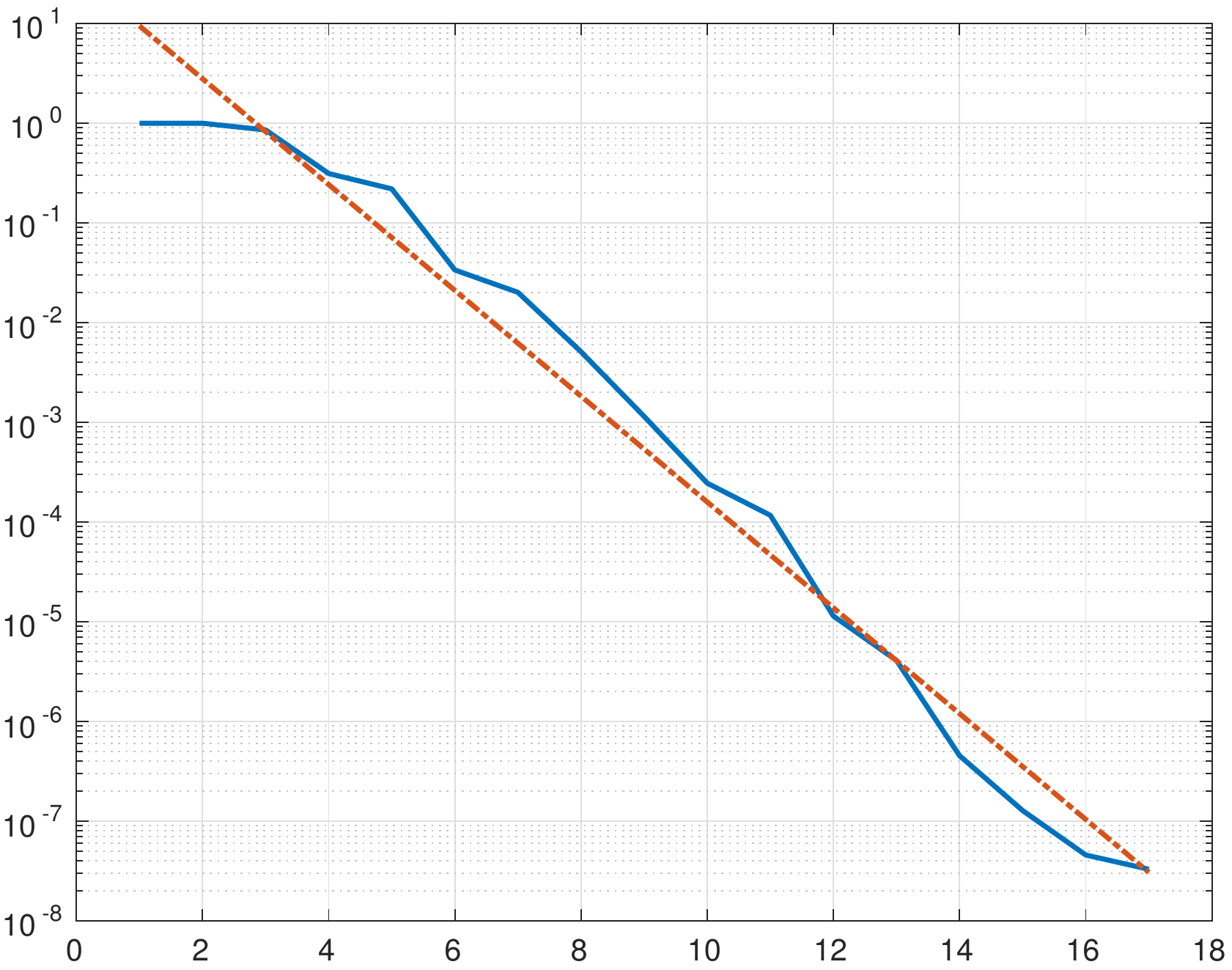}&
\includegraphics[width=0.3 \textwidth]{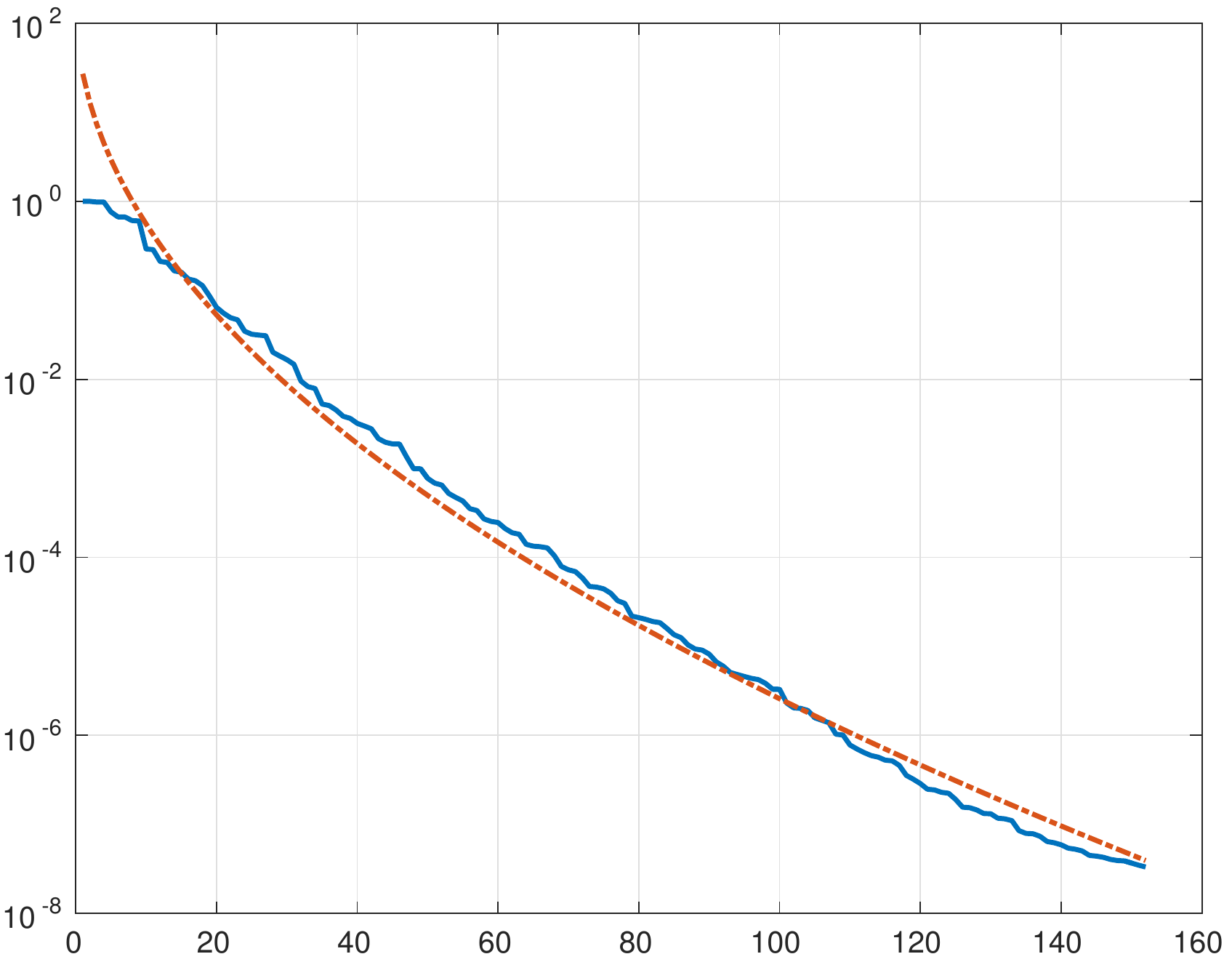}&
\includegraphics[width=0.3 \textwidth]{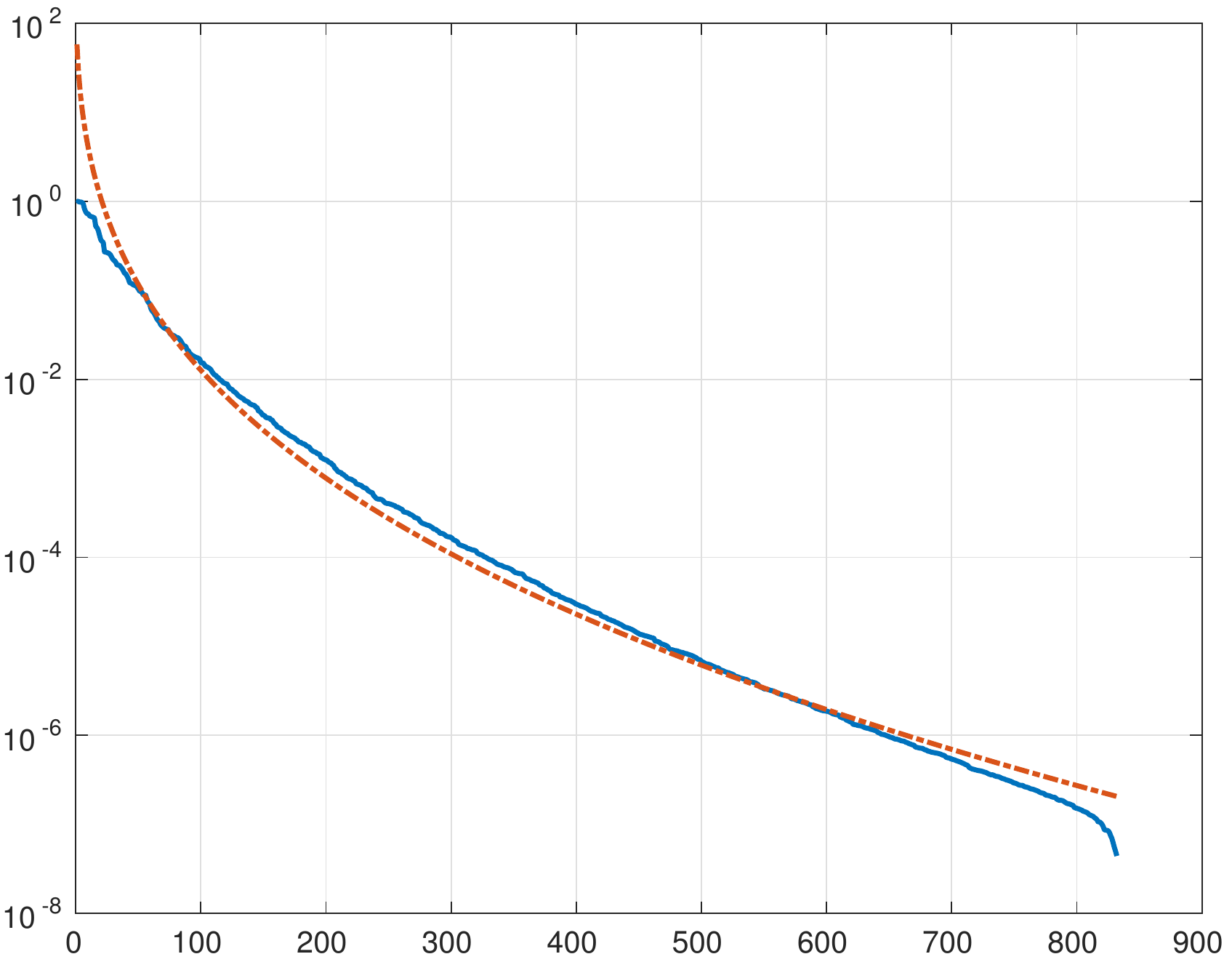}
\end{tabular}
\caption{Expected theoretical rate of convergence (dotted red lines) and computed decay of the Power Function for the $P$-greedy algorithm (solid blue lines), 
with the 
setting described in Section \ref{sec:numerics} and the Gaussian kernel. From left to right: $d =1 , 2, 3$.}\label{fig:gauss_pf}
\end{figure}

Figure \ref{fig:wendland_pf} shows the results of the same experiment for the Wendland kernels. Here we can observe that the theoretical rate of Theorem 
\ref{th:final_bound} seems to be not sharp, and instead the rate of Remark \ref{rem:sobolev_optimal} seems to be valid. We report both rates in the figure, 
computed with scaling coefficients as in Table \ref{tab:wen_coeffs}. These results could be an insight of the optimality of kernel methods in Sobolev spaces, 
where the optimal decay rate can be reached also by greedy methods.

\begin{table}[!hbt]
\centering
\begin{tabular}
{|c|c|c|c|c|c|c|c|c|}
\hline
   &$d = 1$&$d = 2$&$d = 3$&\hspace{0.2cm}&  &$d = 1$&$d = 2$&$d = 3$\\
 \hline
 $\beta = 2$&$0.003$&$0.01$& $0.02 $&&$\beta = 2$&$0.08$&$0.34$& $0.49 $\\
 \hline
 $\beta = 3$&$0.03$&$0.02$&$0.02$&& $\beta = 3$&$0.32$&$0.52$&$0.67$\\
\hline
\end{tabular}
\caption{Estimated coefficient $\hat{c}_1$ for the decay rate of the Power Function with the Wendland kernels for the theoretical rate of 
convergence (left) and the modified rate of convergence (right).}\label{tab:wen_coeffs}
\end{table}

In the same setting, also the fill distance of the selected points is computed. The results are shown in Figure \ref{fig:filldistance}, and they confirm the 
decay rate expected from Corollary \ref{cor:fill_distance}. Also in this case the theoretical rate is scaled by a positive coefficient. Observe that in this 
case the use of a discretized set $\tilde\Omega$ in place of $\Omega$ influences the results of the computations.

\begin{figure}[!hbt]
\centering
\begin{tabular}{ccc}
\includegraphics[width=0.3 \textwidth]{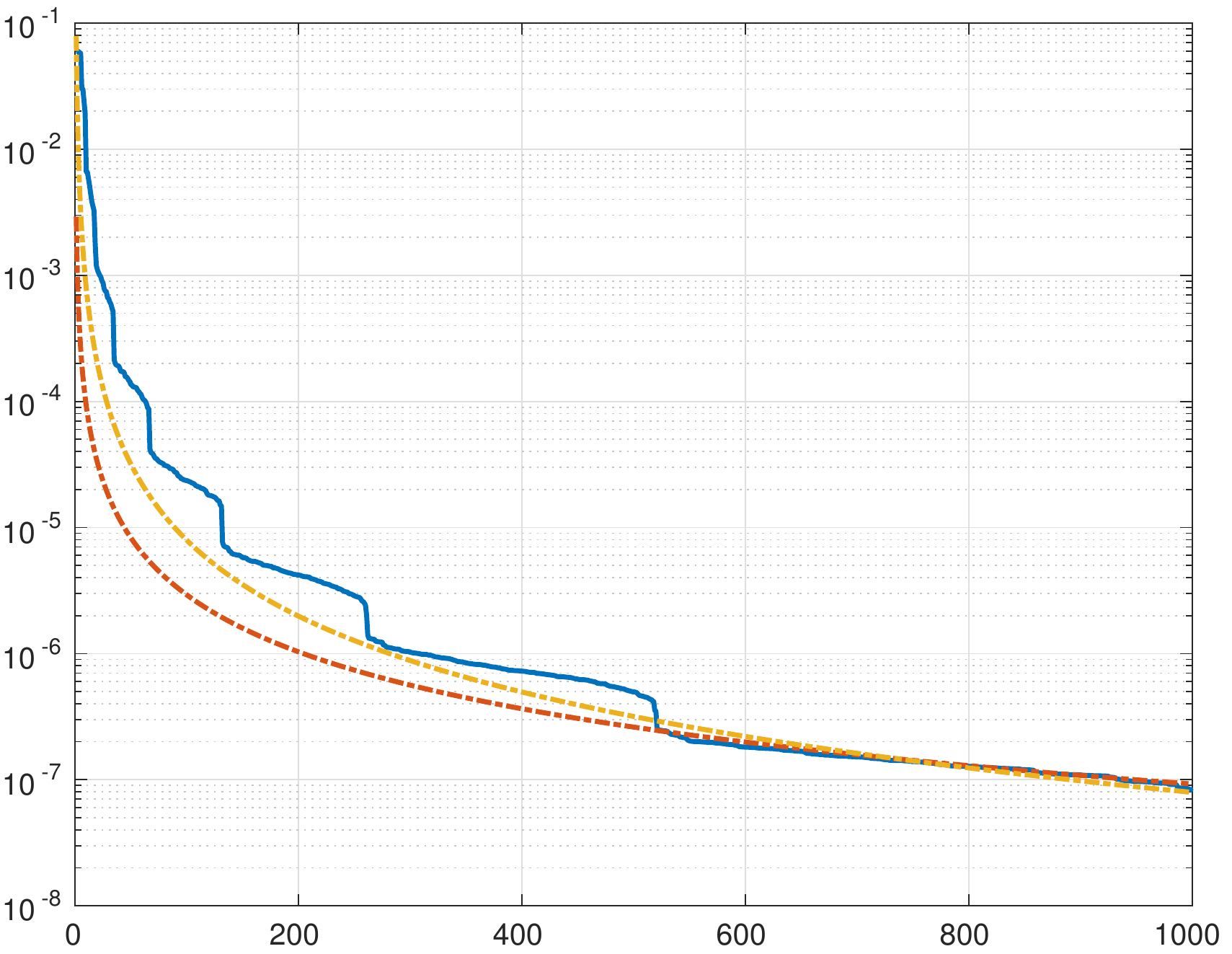}&
\includegraphics[width=0.3 \textwidth]{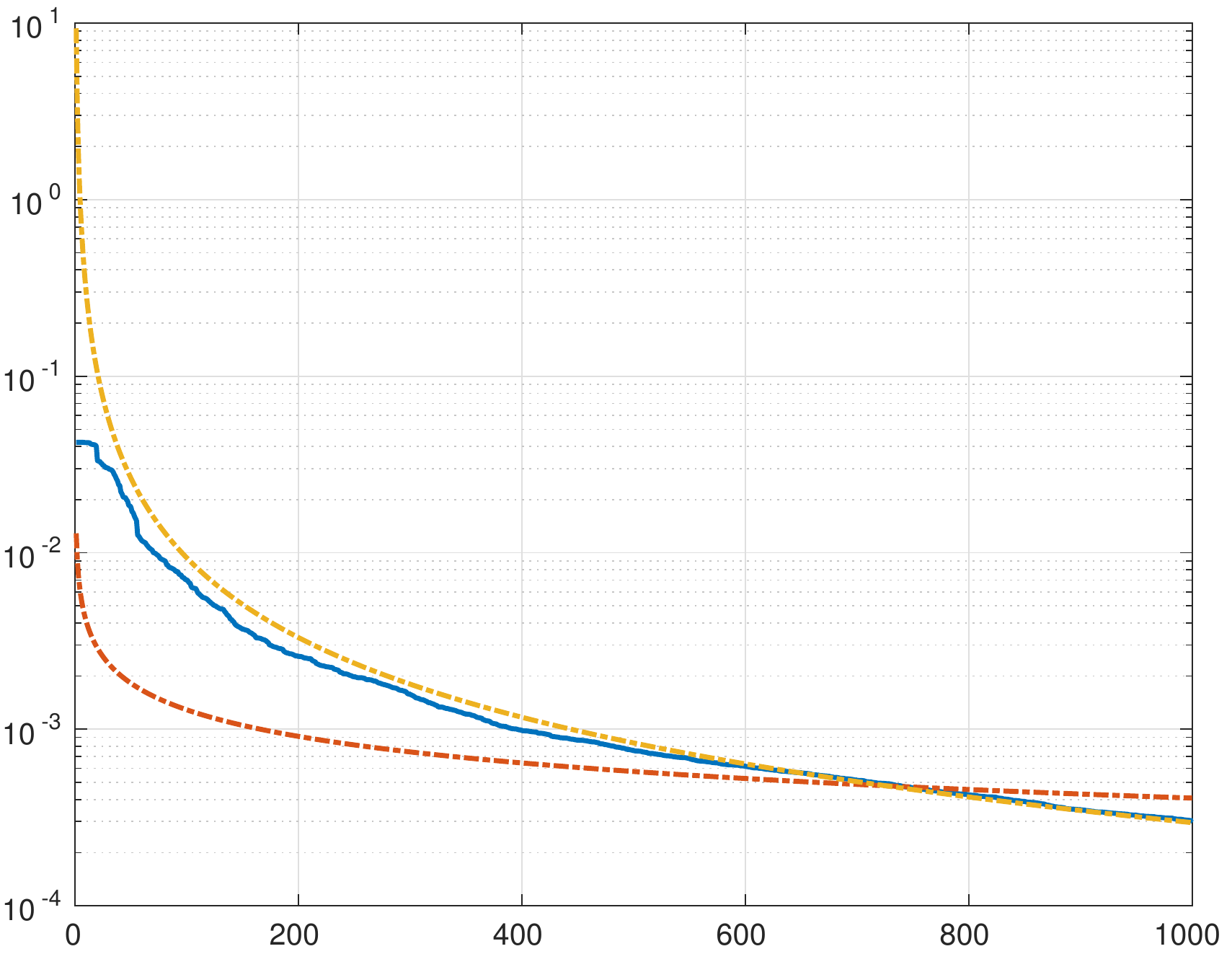}&
\includegraphics[width=0.3 \textwidth]{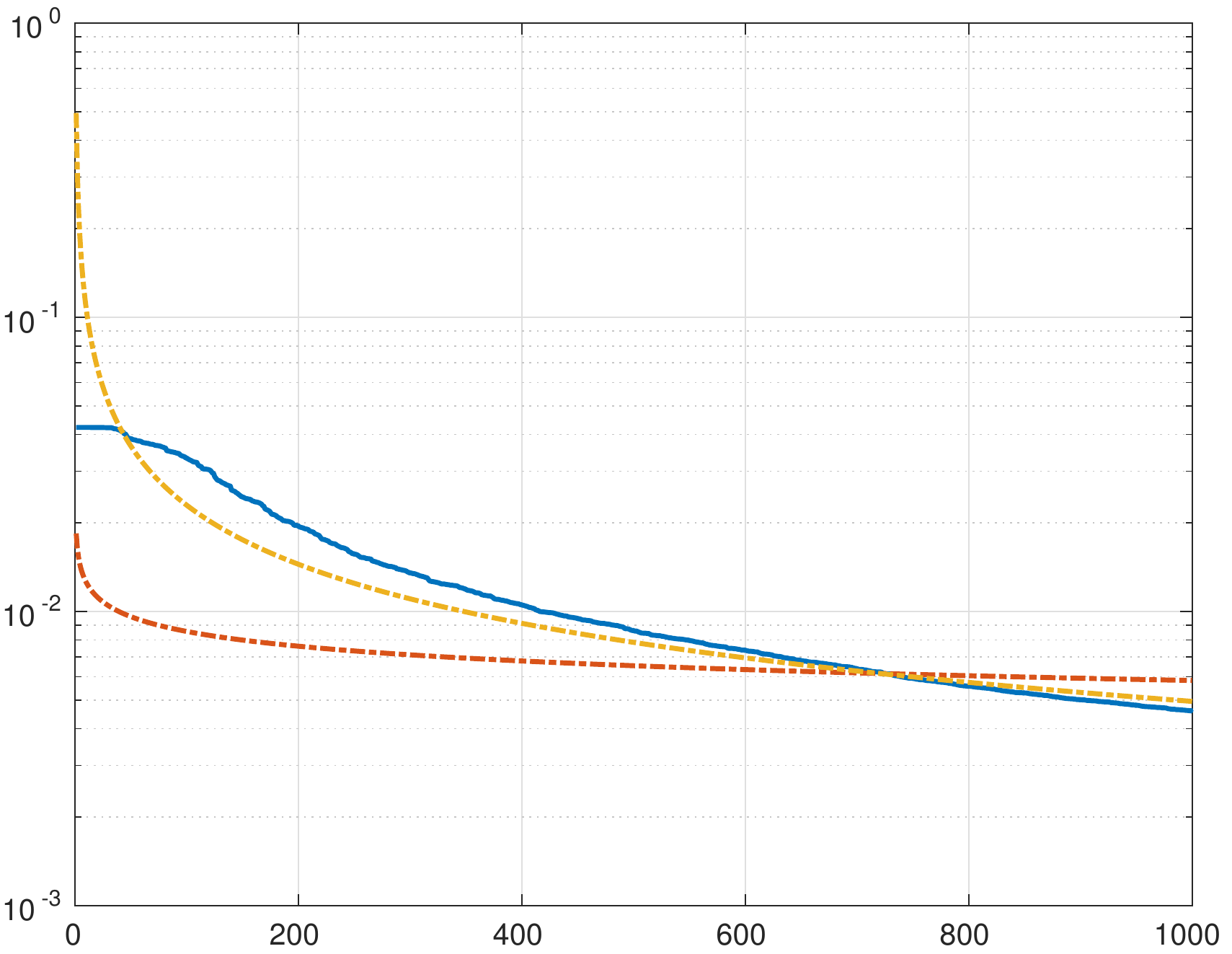}\\
\includegraphics[width=0.3 \textwidth]{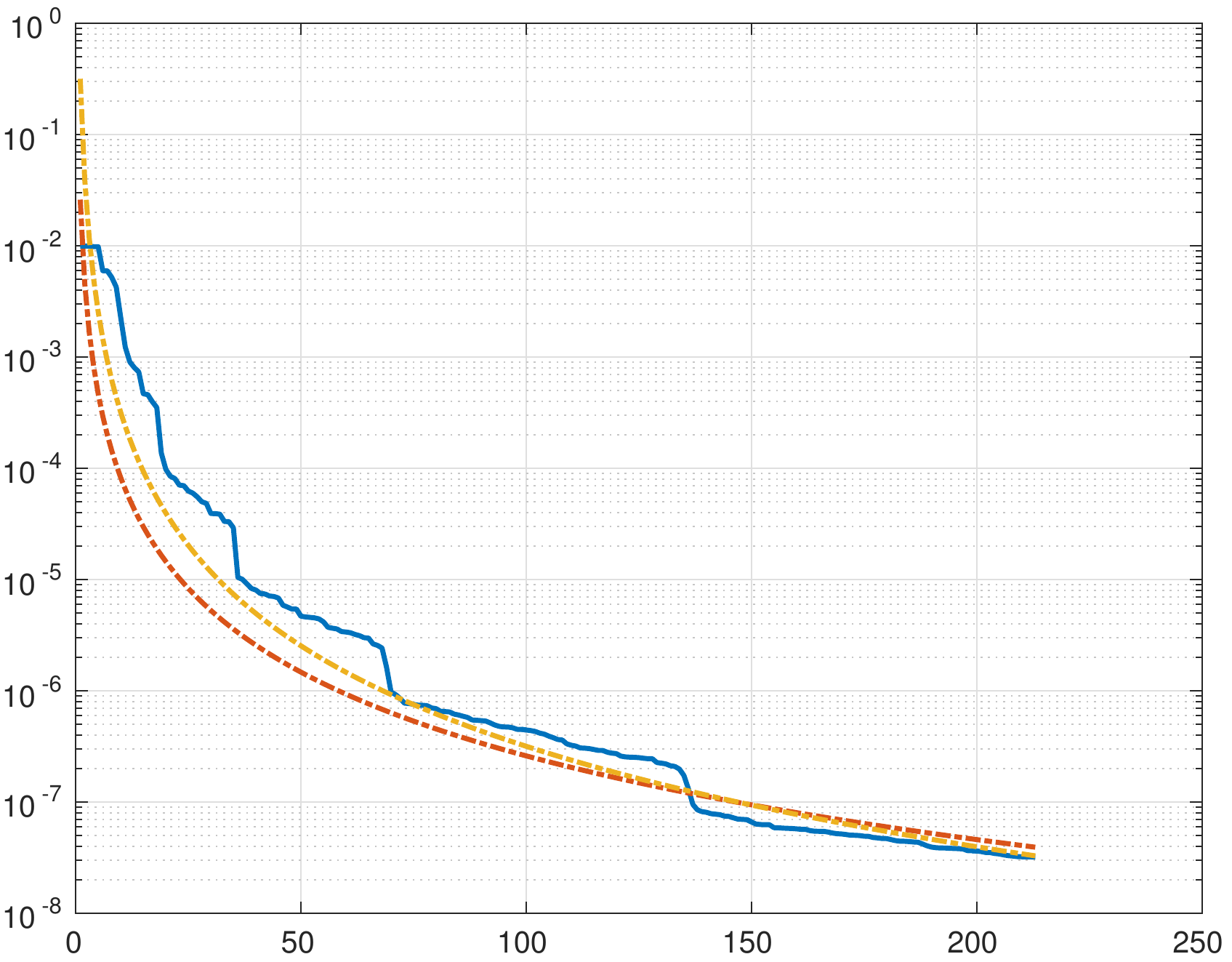}&
\includegraphics[width=0.3 \textwidth]{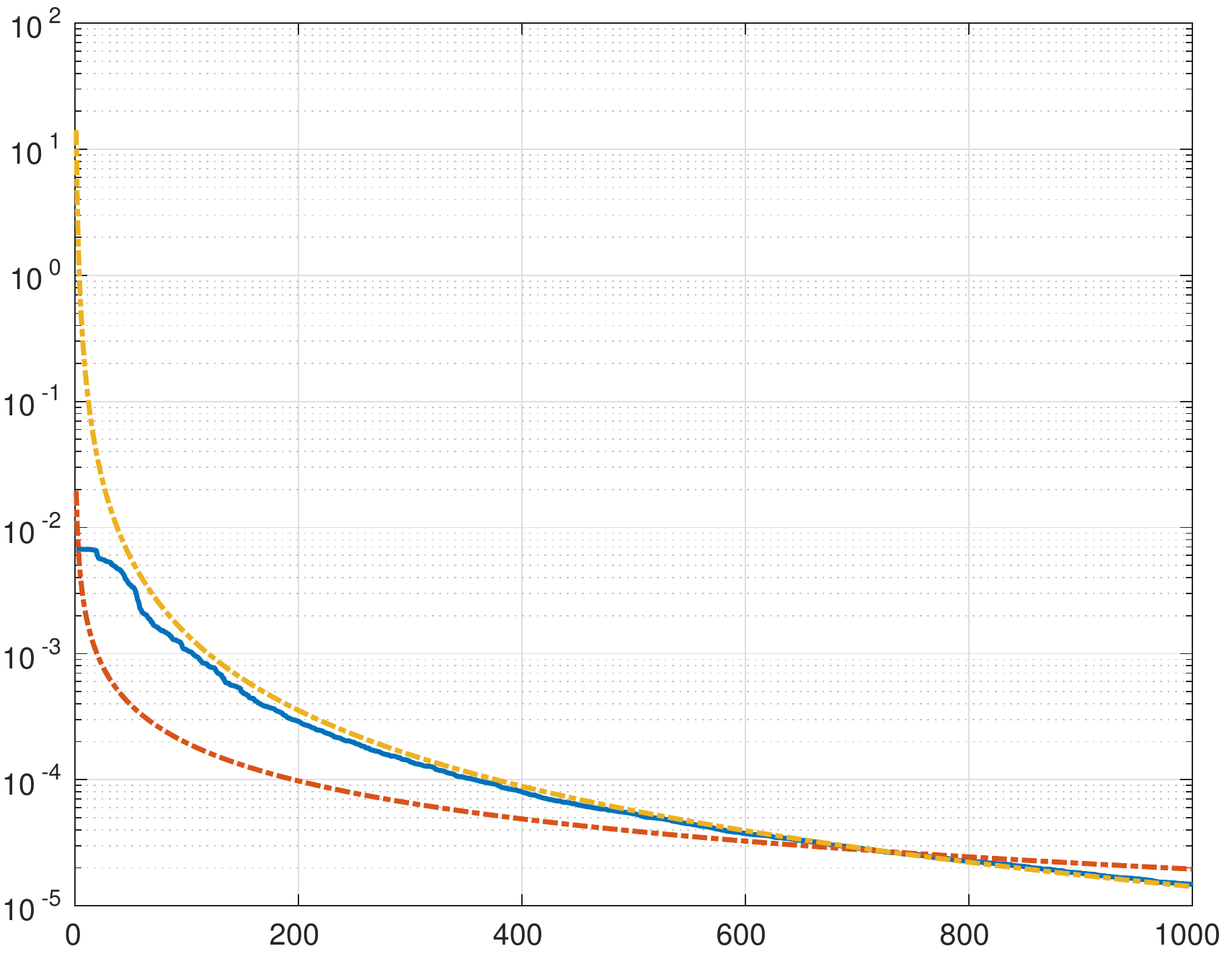}&
\includegraphics[width=0.3 \textwidth]{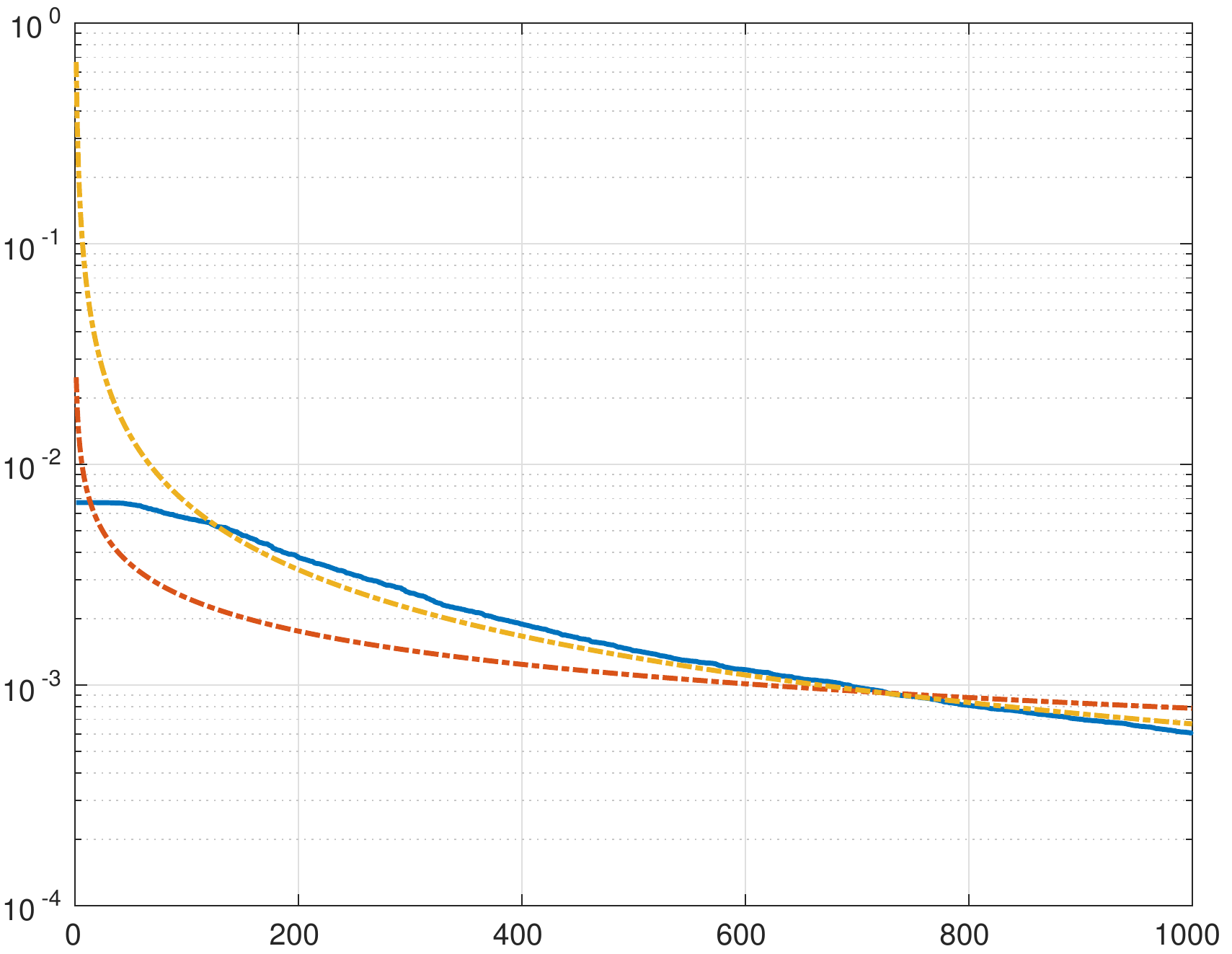}\\
\end{tabular}
\caption{Expected theoretical rate of convergence (dotted red lines), improved rate of convergence (dotted yellow lines), and computed decay of the Power 
Function for the $P$-greedy algorithm (solid blue lines), 
with the setting described in Section \ref{sec:numerics} and the Wendland kernels, with $d =1 , 2, 3$ (from left to right) and $\beta = 2, 3$ (from top to 
bottom).}\label{fig:wendland_pf}
\end{figure}

\begin{figure}[!hbt]
\centering
\begin{tabular}{ccc}
\includegraphics[width=0.3 \textwidth]{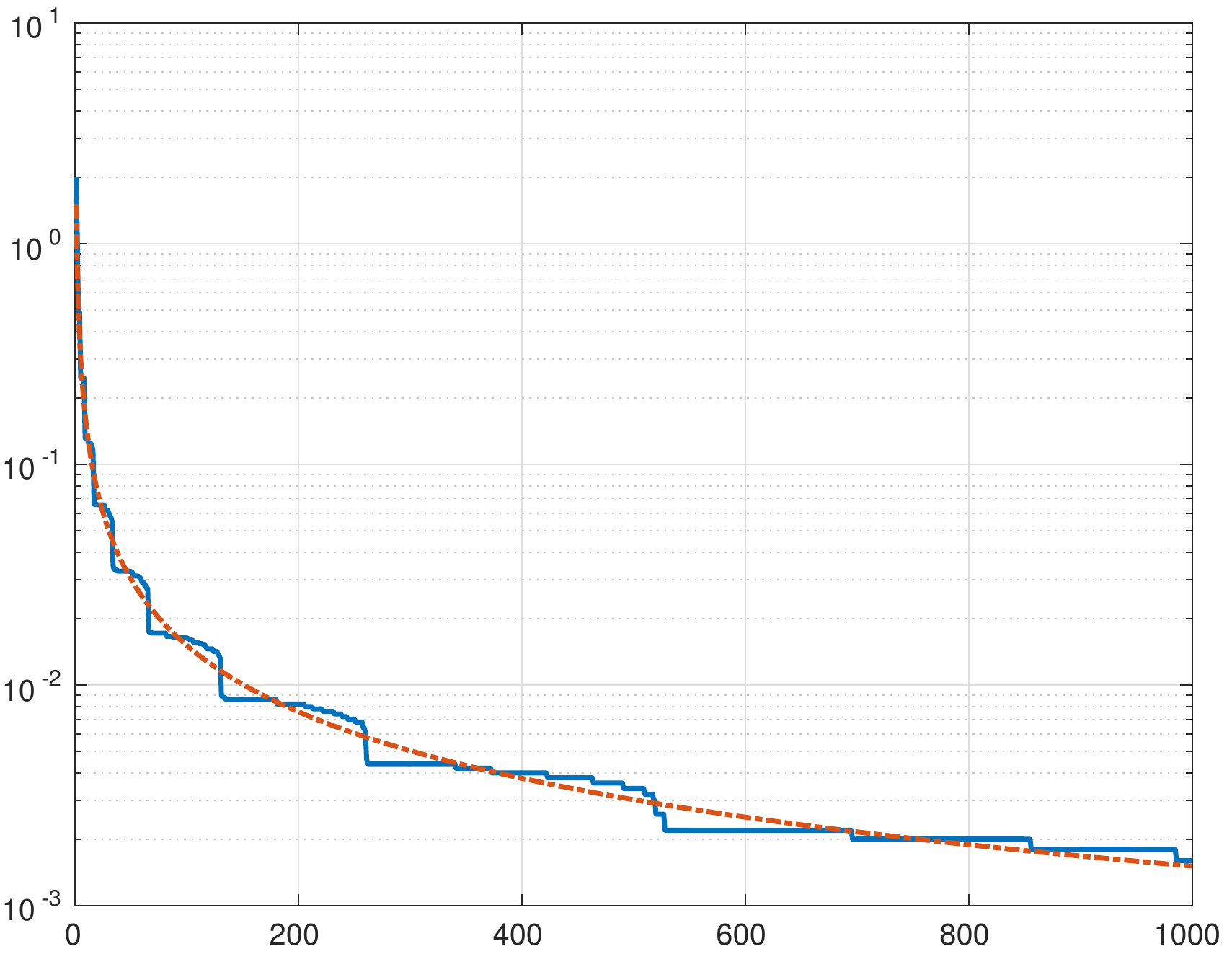}&
\includegraphics[width=0.3 \textwidth]{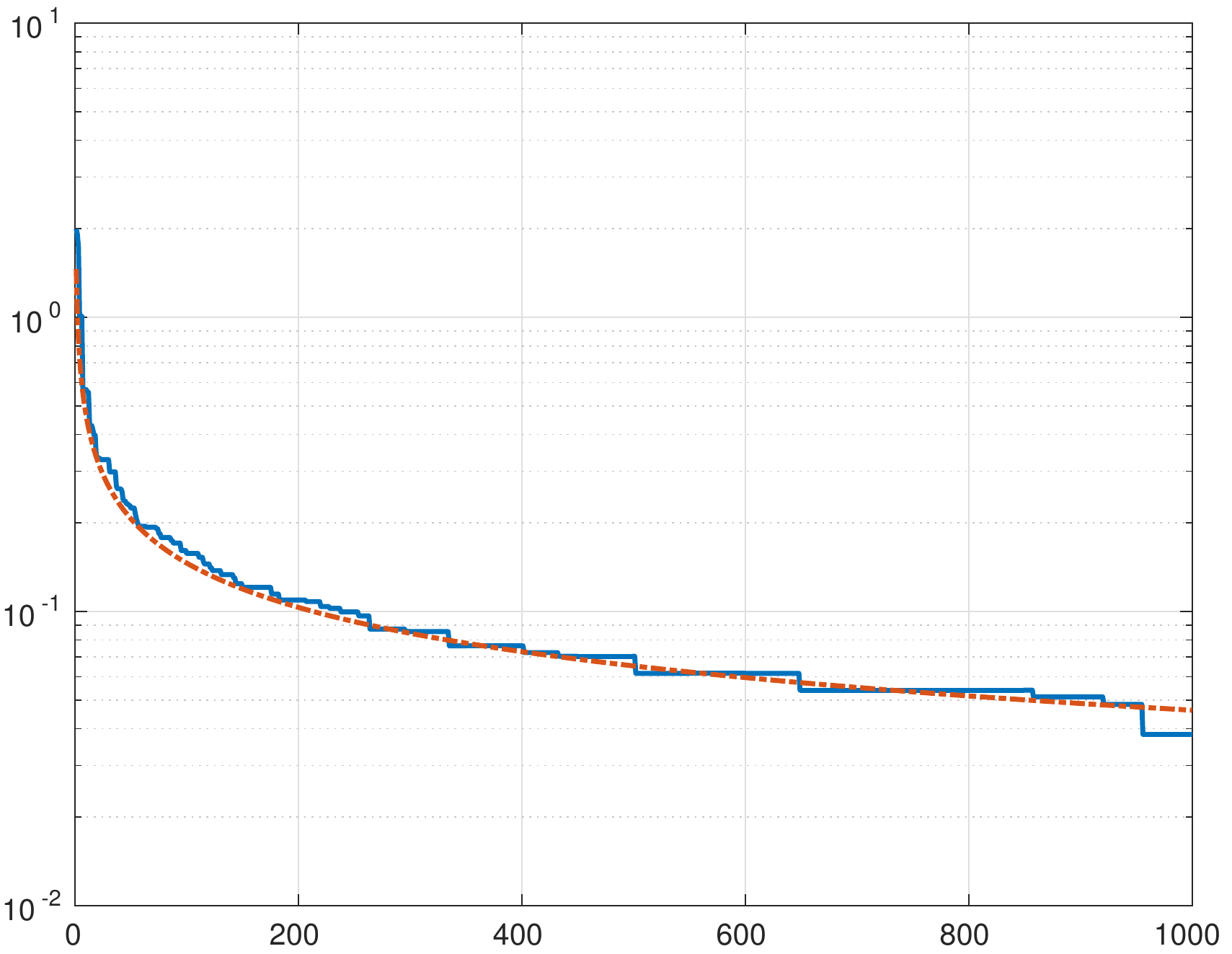}&
\includegraphics[width=0.3 \textwidth]{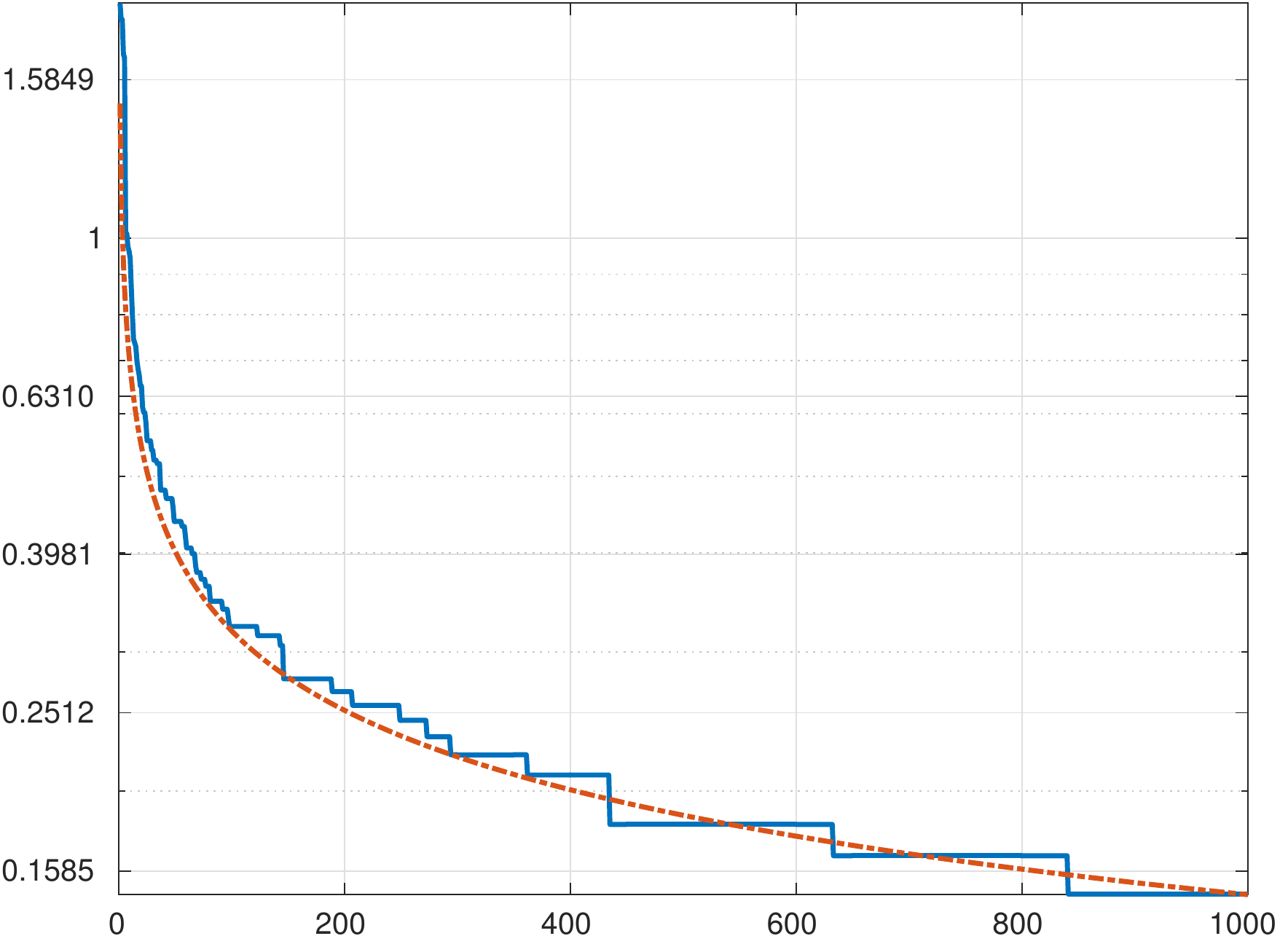}\\
\includegraphics[width=0.3 \textwidth]{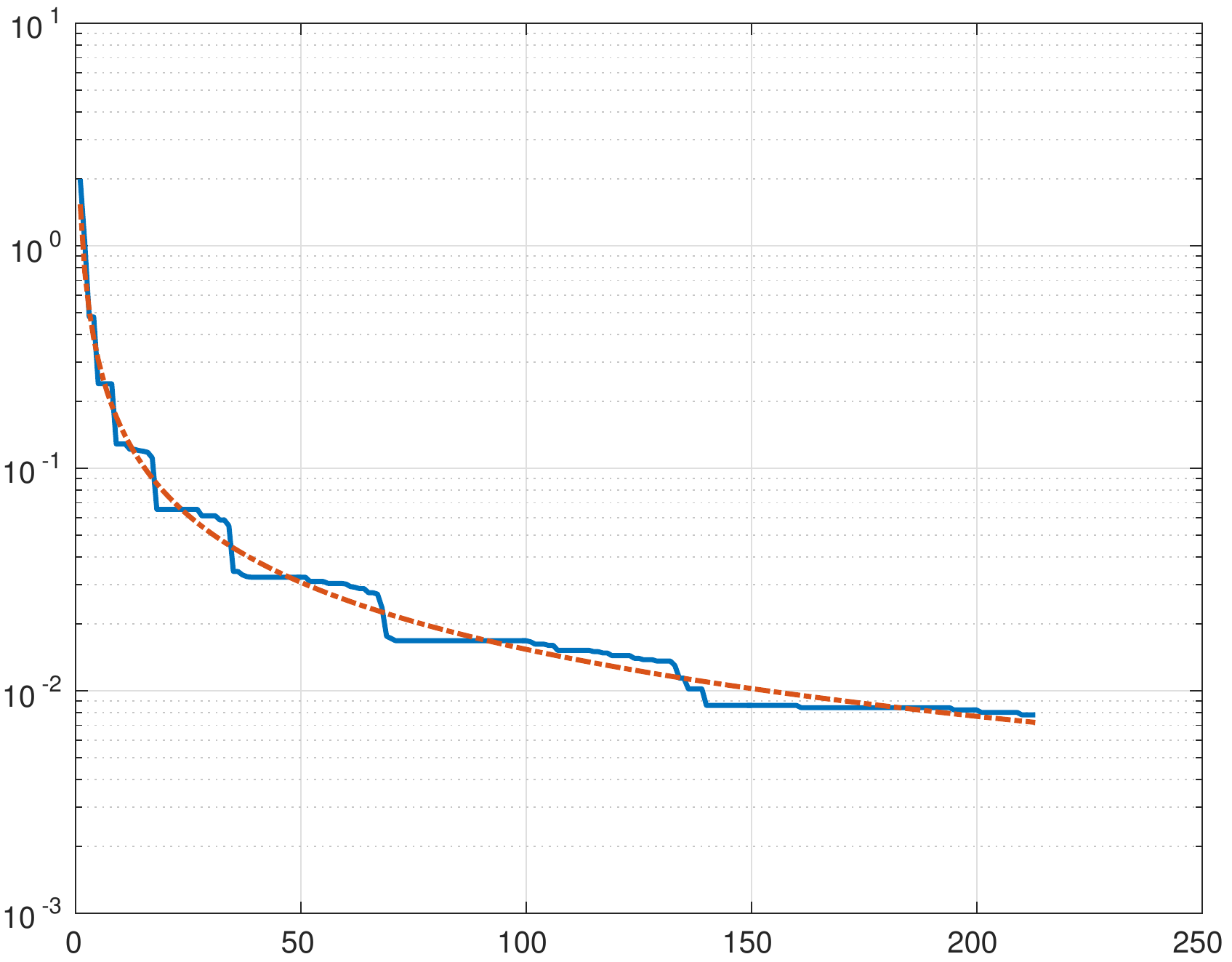}&
\includegraphics[width=0.3 \textwidth]{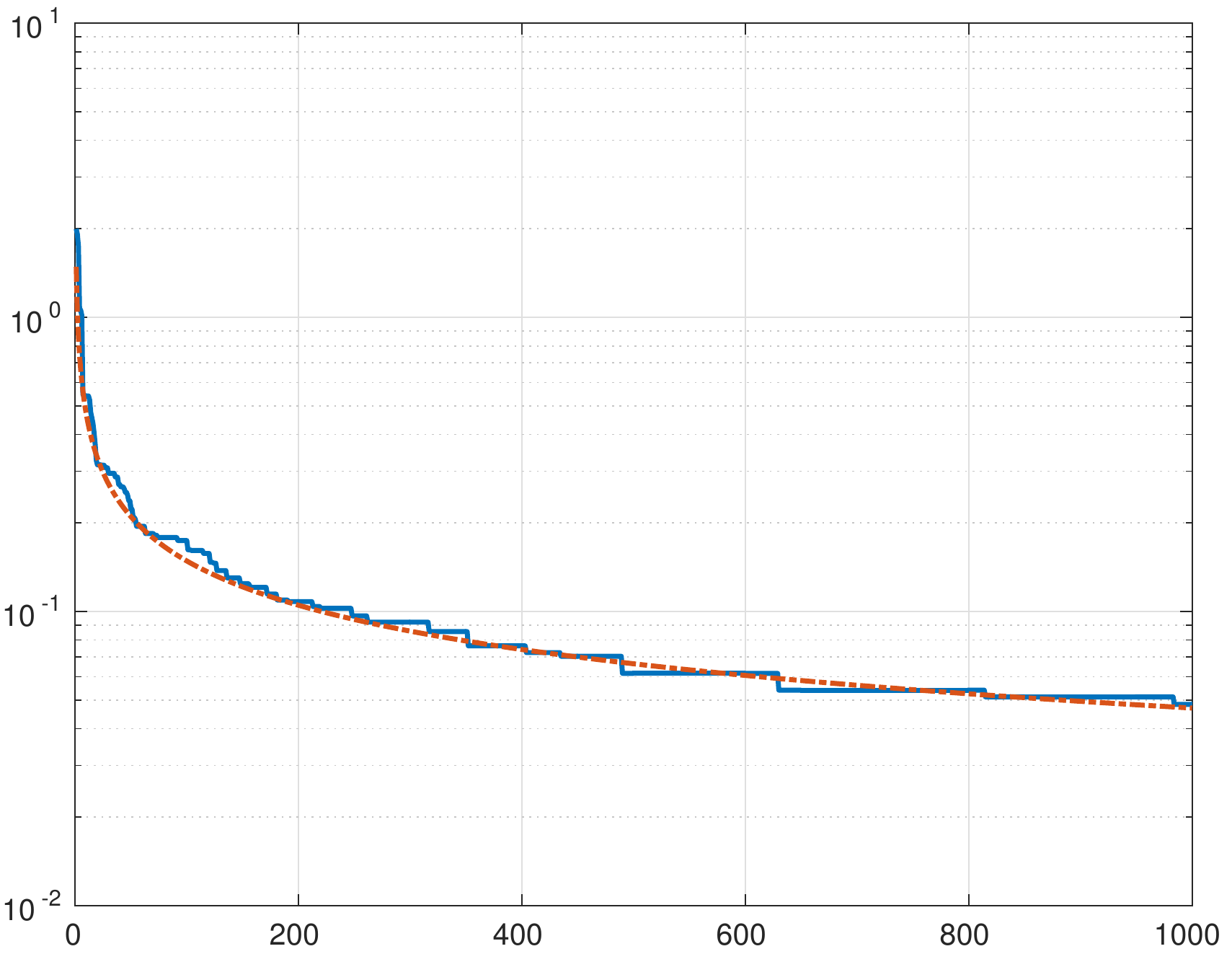}&
\includegraphics[width=0.3 \textwidth]{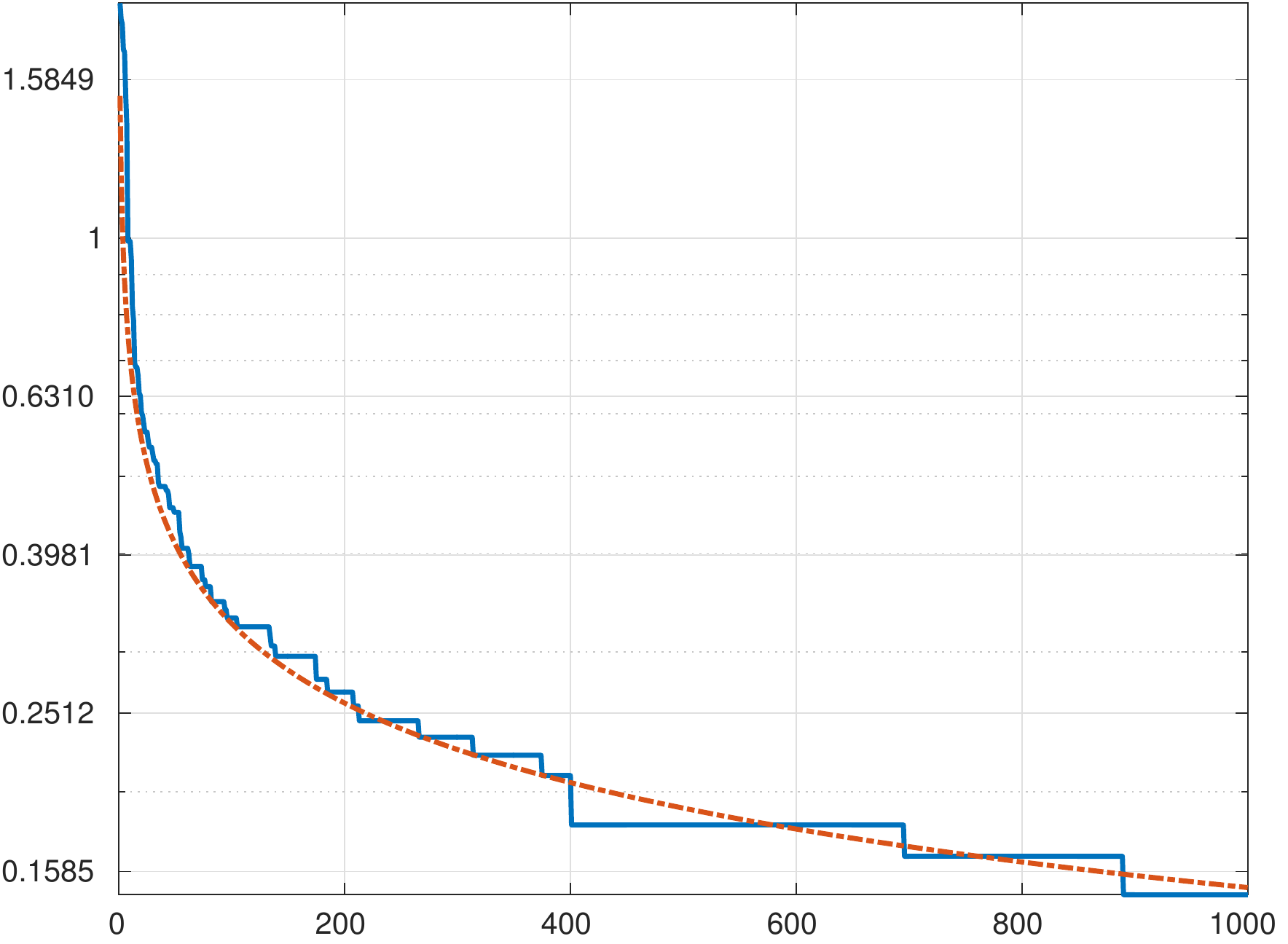}
\end{tabular}
\caption{Expected decay of the fill distance (red dotted lines) and computed decay (solid blue lines), 
with the setting described in Section \ref{sec:numerics} for the Wendland kernels with $\beta = 2, 3$ (from top to bottom) and $d = 1, 2, 3$ (from left to 
right).
}\label{fig:filldistance}
\end{figure}

% \begin{figure}[!hbt]
% \centering
% \begin{tabular}{ccc}
% \includegraphics[width=0.3 \textwidth]{gauss1_fill}&
% \includegraphics[width=0.3 \textwidth]{gauss2_fill}&
% \includegraphics[width=0.3 \textwidth]{gauss3_fill}\\
% \end{tabular}
% \caption{Expected decay of the fill distance (dotted lines) and computed decay (solid lines), 
% with the setting described in Section \ref{sec:numerics} for $d = 1, 2, 3$ (from top to bottom) and $\beta = 1, 2, 3$ (from left to right).
% }\label{fig:gauss_filldistance}
% \end{figure}

\vspace*{1cm}
\textbf{Acknowledgements:} We thank Dominik Wittwar for fruitful discussions.

\bibliographystyle{abbrv}
\bibliography{anm}

\begin{thebibliography}{10}

\bibitem{Binev2011}
P.~Binev, A.~Cohen, W.~Dahmen, R.~DeVore, G.~Petrova, and P.~Wojtaszczyk.
\newblock Convergence rates for greedy algorithms in reduced basis methods.
\newblock {\em SIAM J. Math. Anal.}, 43(3):1457--1472, 2011.

\bibitem{Buhmann2003}
M.~D. Buhmann.
\newblock {\em Radial basis functions: theory and implementations}, volume~12
  of {\em Cambridge Monographs on Applied and Computational Mathematics}.
\newblock Cambridge University Press, Cambridge, 2003.

\bibitem{DeMarchi2005}
S.~De~Marchi, R.~Schaback, and H.~Wendland.
\newblock Near-optimal data-independent point locations for radial basis
  function interpolation.
\newblock {\em Adv. Comput. Math.}, 23(3):317--330, 2005.

\bibitem{DeVore2013}
R.~DeVore, G.~Petrova, and P.~Wojtaszczyk.
\newblock Greedy algorithms for reduced bases in {B}anach spaces.
\newblock {\em Constr. Approx.}, 37(3):455--466, 2013.

\bibitem{Fasshauer2007}
G.~E. Fasshauer.
\newblock {\em Meshfree {A}pproximation {M}ethods with {MATLAB}}, volume~6 of
  {\em Interdisciplinary Mathematical Sciences}.
\newblock World Scientific Publishing Co. Pte. Ltd., Hackensack, {NJ}, 2007.
\newblock With 1 {CD-ROM} (Windows, Macintosh and {UNIX)}.

\bibitem{Fasshauer2015}
G.~E. Fasshauer and M.~McCourt.
\newblock {\em Kernel-{B}ased {A}pproximation {M}ethods Using {MATLAB}},
  volume~19 of {\em Interdisciplinary Mathematical Sciences}.
\newblock World Scientific Publishing Co. Pte. Ltd., Hackensack, {NJ}, 2015.

\bibitem{Jerome1970}
J.~W. Jerome.
\newblock On {$n$}-widths in {S}obolev spaces and applications to elliptic
  boundary value problems.
\newblock {\em J. Math. Anal. Appl.}, 29:201--215, 1970.

\bibitem{Muller2009}
S.~M{\"u}ller and R.~Schaback.
\newblock A {N}ewton basis for kernel spaces.
\newblock {\em J. Approx. Theory}, 161(2):645--655, 2009.

\bibitem{Pazouki2011}
M.~Pazouki and R.~Schaback.
\newblock Bases for kernel-based spaces.
\newblock {\em Journal of Computational and Applied Mathematics}, 236(4):575â
  -- 588, 2011.

\bibitem{Pinkus1985}
A.~Pinkus.
\newblock {\em {$n$}-{W}idths in {A}pproximation {T}heory}, volume~7 of {\em
  Ergebnisse der Mathematik und ihrer Grenzgebiete (3) [Results in Mathematics
  and Related Areas (3)]}.
\newblock Springer-Verlag, Berlin, 1985.

\bibitem{Santin2016a}
G.~Santin and R.~Schaback.
\newblock Approximation of eigenfunctions in kernel-based spaces.
\newblock {\em Advances in Computational Mathematics}, 42(4):973--993, 2016.

\bibitem{Schaback1995}
R.~Schaback.
\newblock Error estimates and condition numbers for radial basis function
  interpolation.
\newblock {\em Adv. Comput. Math.}, 3(3):251--264, 1995.

\bibitem{Schaback2002a}
R.~Schaback and H.~Wendland.
\newblock Approximation by positive definite kernels.
\newblock In M.~Buhmann and D.~Mache, editors, {\em Advanced Problems in
  Constructive Approximation}, volume 142 of {\em International Series in
  Numerical Mathematics}, pages 203--221, 2002.

\bibitem{Wendland1995a}
H.~Wendland.
\newblock Piecewise polynomial, positive definite and compactly supported
  radial functions of minimal degree.
\newblock {\em Advances in Computational Mathematics}, 4(1):389--396, 1995.

\bibitem{Wendland2005}
H.~Wendland.
\newblock {\em Scattered {D}ata {A}pproximation}, volume~17 of {\em Cambridge
  Monographs on Applied and Computational Mathematics}.
\newblock Cambridge University Press, Cambridge, 2005.

\bibitem{Wirtz2013}
D.~Wirtz and B.~Haasdonk.
\newblock A vectorial kernel orthogonal greedy algorithm.
\newblock {\em Dolomites Research Notes on Approximation}, 6:83--100, 2013.
\newblock Proceedings of DWCAA12.

\end{thebibliography}

\end{document}